\documentclass[final]{siamart171218}

\usepackage{amssymb}
\usepackage{amsmath}
\usepackage{mathrsfs}

\usepackage[T1]{fontenc}            %% Polish fonts
\usepackage[utf8]{inputenc}        %% File encoding
\usepackage[polish,british]{babel}  %% Languages (for hyphenation rules), british is prefered

\usepackage{verbatim}               %% I use verbatim mainly for comment enviroment.
\usepackage[all,line]{xy}                %% xy is for diagrams
\usepackage{graphicx}        %% graphicx is for inserting eps or other figures in the text.
\usepackage{url}                    %% for producing links, mainly in the references.

\usepackage{multirow}

\def\CC{{\mathbb C}}

\def\NN{{\mathbb N}}
\def\PP{{\mathbb P}}

\def\RR{{\mathbb R}}
\def\ZZ{{\mathbb Z}}

\def\kk{{\Bbbk}}

\newcommand{\ccB}{{\mathcal{B}}}

\newcommand{\ccM}{{\mathcal{M}}}

\DeclareMathOperator{\codim}{codim}

\DeclareMathOperator{\Gr}{Gr}

\DeclareMathOperator{\Id}{Id}

\DeclareMathOperator{\rk}{rk}

\newcommand{\set}[1]{\left\{#1\right\}}
\newcommand{\fromto}[2]{#1, \dotsc, #2}
\newcommand{\setfromto}[2]{\set{\fromto{#1}{#2}}}

\newcommand{\reduced}[1]{{#1}_{\operatorname{red}}}

\numberwithin{equation}{section}
\newtheorem{thm}[equation]{Theorem}
\newtheorem{prop}[equation]{Proposition}

\newtheorem{lem}[equation]{Lemma}
\newtheorem{cor}[equation]{Corollary}

\newtheorem{problem}{Problem}

\newtheorem{example}[equation]{Example}
\newtheorem{rem}[equation]{Remark}

\newtheorem{defin}[equation]{Definition}

\newtheorem{notation}[equation]{Notation}

\newenvironment{prf}[1][]
  {\medskip\par\noindent{\bf Proof#1. }}
  {\nopagebreak\qed\par\smallskip}

\newcounter{betweenenumi}

\theoremstyle{plain}

\renewcommand{\theenumi}{(\roman{enumi})}

\title{On Strassen's rank additivity for small three-way tensors \thanks{Submitted to the editors
DATE.}}
\author{Jaros\L{}aw Buczy\'{n}ski\thanks{
   Institute of Mathematics of the Polish Academy of Sciences, ul.~\'Sniadeckich 8, 00-656 Warsaw, Poland,
 and Faculty of Mathematics, Computer Science and Mechanics, University of Warsaw, ul.~Banacha 2, 02-097 Warsaw, Poland (\email{jabu@mimuw.edu.pl})}\and 
 Elisa Postinghel\thanks{
Department of Mathematical Sciences, Loughborough University, Leicestershire LE11 3TU, United Kingdom (\email{e.postinghel@lboro.ac.uk})}\and 
Filip Rupniewski\thanks{ 
Institute of Mathematics of the Polish Academy of Sciences, ul.~\'Sniadeckich 8, 00-656 Warsaw, Poland (\email{f.rupniewski@impan.pl})}}
\DeclareMathOperator{\Seg}{Seg}
\date{}

\headers{On Strassen's rank additivity for small three-way tensors}{Jaros{\L}aw Buczy\'{n}ski, Elisa Postinghel, Filip Rupniewski}

\newcommand{\linspan}[1]{\left\langle#1\right\rangle}
\newcommand{\repletion}[1]{{}^{\Re} #1}
\newcommand{\gap}[1]{R(#1) - \dim (#1)}
\newcommand{\bfa}{\mathbf{a}}
\newcommand{\bfb}{\mathbf{b}}
\newcommand{\bfc}{\mathbf{c}}

\newcommand{\bfe}{\mathbf{e}}
\newcommand{\bff}{\mathbf{f}}

\newcommand{\Prime}{\ensuremath{\operatorname{Prime}}}
\newcommand{\Bis}{\ensuremath{\operatorname{Bis}}}
\newcommand{\VL}{\ensuremath{\operatorname{VL}}}
\newcommand{\VR}{\ensuremath{\operatorname{VR}}}
\newcommand{\HL}{\ensuremath{\operatorname{HL}}}
\newcommand{\HR}{\ensuremath{\operatorname{HR}}}
\newcommand{\Mix}{\ensuremath{\operatorname{Mix}}}

\newcommand{\prim}{\mathbf{prime}}
\newcommand{\bis}{\mathbf{bis}}
\newcommand{\vl}{\mathbf{vl}}
\newcommand{\vr}{\mathbf{vr}}
\newcommand{\hl}{\mathbf{hl}}
\newcommand{\hr}{\mathbf{hr}}
\newcommand{\mix}{\mathbf{mix}}

\usepackage{adjustbox} % to scale size of a matrix
\usepackage{color}
\newenvironment{red}{\color{red}}{}
\newcommand{\bred}{\begin{red}}
\newcommand{\ered}{\end{red}}

\newenvironment{blue}{\color{blue}}{}
\newcommand{\bblue}{\begin{blue}}
\newcommand{\eblue}{\end{blue}}

\newenvironment{green}{\color{green}}{}
\newcommand{\bgreen}{\begin{green}}
\newcommand{\egreen}{\end{green}}

\usepackage[textsize=tiny]{todonotes}
\usepackage[customcolors]{hf-tikz} % colorful matrices
\usetikzlibrary{patterns}
\usetikzlibrary{matrix,decorations.pathreplacing,calc}
\newcommand{\ts}{\otimes}
\newcommand{\br}{{\underline{R}}}
\DeclareMathOperator{\adj}{adj}

\begin{document}

\maketitle

\begin{abstract}
We address the problem of the additivity of the tensor rank.
That is, for two independent tensors we study if the rank of their direct sum is equal to the sum of their individual ranks.
A positive answer to this problem was previously known as Strassen's conjecture %(which stated that the answer is always positive)
until recent counterexamples were proposed by Shitov. The latter are not very explicit, and they are only known to exist asymptotically for very large tensor spaces.
In this article we prove that for some small three-way tensors the additivity holds.
For instance, if the rank of one of the tensors is at most 6, then the additivity holds.
Or, if one of the tensors lives in $\CC^k\otimes \CC^3\otimes \CC^3$ for any $k$,
  then the additivity also holds.
More generally, if one of the tensors is concise and its rank is at most 2
more than the dimension of one of the linear spaces, then additivity holds.
In addition we also treat some cases of the additivity of the border rank of such tensors.
In particular, we show that the additivity of the border rank holds if the direct sum tensor is contained in 
   $\CC^4\otimes \CC^4\otimes \CC^4$.
Some of our results are valid over an arbitrary base field.
\end{abstract}

\medskip
{\footnotesize
%\noindent\textbf{addresses:} \\
%J.~Buczy\'nski, \nolinkurl{jabu@mimuw.edu.pl},
%   Institute of Mathematics of the Polish Academy of Sciences, ul.~\'Sniadeckich 8, 00-656 Warsaw, Poland,
% and Faculty of Mathematics, Computer Science and Mechanics, University of Warsaw, ul.~Banacha 2, 02-097 Warszawa, Poland\\
%E.~Postinghel, \nolinkurl{e.postinghel@lboro.ac.uk},
%Department of Mathematical Sciences, Loughborough University, Leicestershire LE11 3TU, United Kingdom.
%   \\
%F.~Rupniewski, \nolinkurl{f.rupniewski@impan.pl}, 
%Institute of Mathematics of the Polish Academy of Sciences, ul.~\'Sniadeckich 8, 00-656 Warsaw, Poland.


\section{Introduction}\label{sect_intro}

The matrix multiplication is a bilinear map $\mu_{i,j,k}\colon \ccM^{i\times j} \times \ccM^{j\times k} \to \ccM^{i \times k}$, 
   where $\ccM^{l \times m}$ 
   is the linear space of $l\times m $ matrices with coefficients in a field $\kk$.
In particular, $\ccM^{l \times m} \simeq \kk^{l \cdot m}$,
   where $\simeq$ denotes an isomorphism of vector spaces.
   We can interpret $\mu_{i,j,k}$ as a three-way tensor 
\[
   \mu_{i,j,k}  \in (\ccM^{i \times j})^* \otimes (\ccM^{j \times k})^*\otimes \ccM^{i \times k}.
\]
Following the discoveries of Strassen \cite{strassen_gaussian_elimination_is_not_optimal}, scientists started to wonder
   what is the minimal number of multiplications required to calculate the product of two matrices $MN$, 
  for any $M\in \ccM^{i \times j}$ and $N \in \ccM^{j \times k}$.
This is a question about the \emph{tensor rank} of $\mu_{i,j,k}$.

Suppose $A$, $B$, and $C$ are finite dimensional vector spaces over $\kk$.
A \emph{simple tensor} is an element of the tensor space $A\otimes B \otimes C$ which can be written as $a\otimes b \otimes c$ for some $a\in A$, $b\in B$, $c\in C$.
The rank of a tensor $p\in A\otimes B \otimes C$ is the minimal number $R(p)$ of simple tensors needed, 
  such that $p$ can be expressed as a linear combination of simple tensors.
Thus $R(p)=0$ if and only if $p=0$, and $R(p)=1$ if and only if $p$ is a simple tensor.
In general, the higher the rank is, the more complicated $p$ ``tends'' to be.
In particular, the minimal number of multiplications needed to calculate $MN$ as above is equal to $R(\mu_{i,j,k})$. 
See for instance \cite{comon_tensor_decompositions_survey}, \cite{landsberg_tensorbook},
   \cite{carlini_grieve_oeding_four_lectures_on_secants} and references therein
   for more details and further motivations to study tensor rank.

Our main interest in this article is in the \emph{additivity} of the tensor rank.
Going on with the main example, given arbitrary four matrices 
  $M'\in \ccM^{i' \times j'}$, $N' \in \ccM^{j' \times k'}$, $M''\in \ccM^{i'' \times j''}$, $N'' \in \ccM^{j'' \times k''}$,
  suppose we want to calculate both products $M' N'$ and $M''N''$ simultaneously.
What is the minimal number of multiplications needed to obtain the result? 
Is it equal to the sum of the ranks $R(\mu_{i',j',k'}) + R(\mu_{i'',j'',k''})$?
More generally, the same question can be asked for  arbitrary tensors. 
If we are given two tensors in independent vector spaces, is the rank of their sum equal to the sum of their ranks?
A positive answer to this question was widely known as Strassen's Conjecture \cite[p.~194, \S4, Vermutung~3]{strassen_vermeidung_von_divisionen}, \cite[\S5.7]{landsberg_tensorbook}, 
until it was disproved by Shitov \cite{shitov_counter_example_to_Strassen}.

\begin{problem}[Strassen's additivity problem]\label{prob_strassen}
Suppose $A = A' \oplus A''$, $B = B' \oplus B''$, and $C = C' \oplus C''$, where all $\fromto{A}{C''}$ are finite dimensional vector spaces over a field $\kk$.
Pick $p' \in A' \otimes B' \otimes C'$ and $p'' \in A'' \otimes B'' \otimes C''$ 
  and let $p= p' + p''$, which we will write as $p= p'\oplus p''$.
Does the following equality hold
\begin{equation}\label{equ_additivity}
  R(p) = R(p') + R(p'')?
\end{equation}
\end{problem}

In this article we address several cases of Problem~\ref{prob_strassen} and its gen\-er\-al\-isa\-tions.
It is known that if one of the vector spaces $A'$, $A''$, $B'$, $B''$, $C'$, $C''$ is at most two dimensional,
    then the additivity of the tensor rank \eqref{equ_additivity} holds: 
    see \cite{jaja_takche_Strassen_conjecture} for the original proof
    and Section~\ref{sec_hook_shaped_spaces} for a discussion of more recent approaches.
One of our results includes the next case, that is, if say $\dim B''=\dim C'' = 3$, then \eqref{equ_additivity} holds. The following theorem summarises our main results. 

\begin{thm}\label{thm_additivity_rank_intro}
  Let $\kk$ be any base field 
     and let $A'$, $A''$, $B'$, $B''$, $C'$, $C''$ be vector spaces over $\kk$.
  Assume $p' \in A' \otimes B' \otimes C'$ and $p'' \in A'' \otimes B'' \otimes C''$ and let
  \[
      p = p' \oplus p'' \in (A'\oplus A'') \otimes (B'\oplus B'') \otimes(C'\oplus C'').
  \]
  If at least one of the following conditions holds, 
     then the additivity of the rank holds for $p$, that is, $R(p) = R(p') + R(p'')$:
  \begin{itemize}
   \item $\kk=\CC$ or $\kk=\RR$ (complex or real numbers) and $\dim B''\le 3$ and $\dim C'' \le 3$.
   \item $R(p'')\le \dim A'' +2$ and $p''$ is not contained in $\tilde{A''} \otimes B'' \otimes C''$
          for any linear subspace $\tilde{A''} \subsetneqq A''$
          (this part of the statement is valid for any field $\kk$).
   \item $\kk=\RR$ or $\kk$ 
            is an algebraically closed field of characteristic $\ne 2$ 
            and $R(p'')\le 6$. 
  \end{itemize}
  Analogous statements hold if we exchange the roles of $A$, $B$, $C$ and/or of ${}'$ and ${}''$. 
\end{thm}

The theorem summarises the content of Theorems~\ref{thm_additivity_rank_plus_2}--\ref{thm_additivity_rank_6}
   proven in Section~\ref{sect_proofs_of_main_results_on_rank}.

\begin{rem}
   Although most of our arguments are characteristic free, 
     we partially rely on some earlier results which 
     often are proven only over the fields of the
     complex or the real numbers, or other special fields.
   Specifically, we use upper bounds on the maximal rank 
     of small tensors, 
     such as \cite{bremner_hu_Kruskal_theorem} or \cite{sumi_miyazaki_sakata_maximal_tensor_rank}.
   See Section~\ref{sect_proofs_of_main_results_on_rank}
     for a more detailed discussion.
   In particular, the consequence of the proof of
     Theorem~\ref{thm_additivity_rank_6} is that if
     (over any field $\kk$) there are  $p'$ and $p''$ 
     such that $R(p'')\le 6$ and $R(p'\oplus p'')< R(p')+R(p'')$, 
     then $p''\in \kk^3\otimes \kk^3 \otimes \kk^3$ and $R(p'')=6$.
   In \cite{bremner_hu_Kruskal_theorem}
     it is shown that if $\kk=\ZZ_2$ (the field with two elements),
     then such tensors $p''$ with $R(p'')=6$ exist.
\end{rem}

Some other cases of additivity were shown in \cite{feig_winograd_strassen}.
Another variant of Problem~\ref{prob_strassen} asks the same question in the setting of symmetric tensors 
   and the symmetric tensor rank, or equivalently, for homogeneous polynomials and their Waring rank.
No counterexamples to this version of the problem are yet known,
   while some partial positive results are described in 
   \cite{carlini_catalisano_chiantini_symmetric_Strassen_1},
   \cite{carlini_catalisano_chiantini_geramita_woo_symmetric_Strassen_2},
   \cite{carlini_catalisano_oneto_Strassen},
   \cite{casarotti_massarenti_mella_Comon_and_Strassen}, 
   and \cite{teitler_symmetric_strassen}.
   Possible ad hoc extensions to the symmetric case of the techniques and results obtained in this article are subject of a follow-up research.

Next we turn our attention to the \emph{border rank}. 
Roughly speaking, over the complex numbers, a tensor $p$ has border rank at most $r$, 
  if and only if it is a limit of tensors of rank at most $r$.
The border rank of $p$ is denoted by $\br(p)$.
One can pose the analogue of Problem~\ref{prob_strassen} for the border rank:
  for which tensors  $p' \in A' \otimes B' \otimes C'$ and $p'' \in A'' \otimes B'' \otimes C''$ 
  is the border rank additive, that is, $\br(p'\oplus p'') = \br(p')+\br(p'')$?
  
In general, the answer is negative; in fact there exist examples for which $\br(p'\oplus p'') <\br(p')+ \br (p'')$:
Sch\"onhage \cite{schonhage_matrix_multiplication} proposed a family of counterexamples amongst which the smallest is
\[
  \br(\mu_{2,1,3})=6,\quad \br(\mu_{1,2,1})=2, \quad \br(\mu_{2,1,3}\oplus \mu_{1,2,1})=7,
\]
see also \cite[\S11.2.2]{landsberg_tensorbook}.

Nevertheless, one may be interested in special cases of the problem. We describe one instance suggested by J.~Landsberg 
   (private communication, also mentioned during his lectures at Berkeley in 2014).
\begin{problem}[Landsberg]
   Suppose $A', B', C'$ are vector spaces and $A''\simeq B'' \simeq C'' \simeq \CC$. 
   Let $p'\in A'\otimes B' \otimes C'$ be any tensor and $p'' \in A''\otimes B'' \otimes C''$ be a non-zero tensor.
   Is $\br(p'\oplus p'') > \br(p')$?
\end{problem}

Another interesting question is what is the smallest counterexample to the additivity of the border rank?
The example of Sch\"onhage lives in $\CC^{2+2}\otimes \CC^{3+2}\otimes \CC^{6+1}$,
   that is, it requires using a seven dimesional vector space.
Here we show that if all three spaces $A$, $B$, $C$ have dimensions at most $4$, 
   then it is impossible to find a counterexample to the additivity of the border rank.
\begin{thm}\label{thm_additivity_border_rank_in_dimension_4}
   Suppose $A', A'', B', B'', C',C''$ are vector spaces over $\CC$ 
      and $A=A'\oplus A''$, $B=B'\oplus B''$, and $C=C'\oplus C''$.
   If $\dim A, \dim B, \dim C \le 4$, then 
      for any $p' \in A' \otimes B' \otimes C'$ and $p'' \in A'' \otimes B'' \otimes C''$ the 
      additivity of the border rank holds:
      \[
         \br(p'\oplus p'') = \br(p')+\br(p'').
      \]
\end{thm}
We prove the theorem in Section~\ref{sect_border_rank} 
   as Corollary~\ref{cor_additivity_brank_very_small_a_b_c}, Propositions~\ref{prop_br_case_3_2_2_1_2_2} 
   and~\ref{prop_br_case_3_3_3_1_1_1}, which in fact cover a wider variety of cases.

\subsection{Overview}
 
In this article, for the sake of simplicity, we mostly restrict our presentation to the case of three-way tensors,
  even though some intermediate results hold more generally.
In Section~\ref{sec_basics} we introduce the notation and review known methods about tensors in general.
We review the translation of the rank and border rank of three-way tensors into statements about linear spaces of matrices. 
In Proposition~\ref{prop_bound_on_rank_for_non_concise_decompositions_for_vector_spaces} we explain that any decomposition that uses elements outside of the minimal tensor space containing a given tensor must involve more terms
%be appropriately longer 
than the rank of that tensor. 
In Section~\ref{sect_dir_sums} we present the notation related to the direct sum tensors 
   and we prove the first results on the additivity of the tensor rank.
In particular, we slightly generalise the proof of the additivity of the rank when one of 
   the tensor spaces has dimension $2$.
In Section~\ref{sec_rank_one_matrices_and_additive_rank} we analyse 
   rank one matrices contributing to the minimal decompositions of tensors, 
   and we distinguish seven types of such matrices. 
Then we show that to prove the additivity of the tensor rank one can get rid of two of those types, 
   that is, we can produce a smaller example, which does not have these two types, 
   but if the additivity holds for the smaller one, then it also holds for the original one.
This is the core observation to prove the main result, Theorem~\ref{thm_additivity_rank_intro}.
Finally, in Section~\ref{sect_border_rank} we analyse the additivity of the border rank for small tensor spaces. 
For most of the possible splittings of the triple $A=\CC^4 =A'\oplus A''$, 
  $B=\CC^4 =B'\oplus B''$, $C=\CC^4 =C'\oplus C''$, 
  there is an easy observation (Corollary~\ref{cor_additivity_brank_very_small_a_b_c}) 
  proving the additivity of the border rank.
The remaining two pairs of triples are treated by more advanced methods,
  involving in particular the Strassen type equations for secant varieties.
We conclude the article with a brief discussion of the potential analogue of 
  Theorem~\ref{thm_additivity_border_rank_in_dimension_4} for 
  $A=B=C=\CC^5$.

\section{Ranks and slices}\label{sec_basics}

This section reviews the notions of rank, border rank, slices, conciseness. Readers that are familiar to these concepts may easily skip this section.
The main things to remember from here are Notation~\ref{notation_V_Seg} and 
  Proposition~\ref{prop_bound_on_rank_for_non_concise_decompositions_for_vector_spaces}.

Let $\fromto{A_1, A_2}{A_d}$, $A$, $B$, $C$, and $V$ be finite dimensional vector spaces over 
% an algebraically closed 
a field $\kk$.
Recall a tensor $s\in  A_1 \otimes A_2 \otimes \dotsb \otimes A_d$ is \emph{simple} if and only if it can be written as $a_1 \otimes a_2 \otimes \dotsb\otimes a_d$ 
   with $a_i \in A_i$. Simple tensors will also be referred to as \emph{rank one tensors} throughout this paper. 
If $P$ is a subset of $V$,  we denote by $\langle P \rangle$ its linear span.
If $P=\setfromto{p_1}{p_r}$ is a finite subset,  we will write
  $\langle\fromto{p_1}{p_r}\rangle$ rather than  $\langle\setfromto{p_1}{p_r}\rangle$ to simplify notation.
   
\begin{defin}
   Suppose $W\subset A_1 \otimes A_2 \otimes \dotsb \otimes A_d$ is a linear subspace of the tensor product space.
   We define $R(W)$, \emph{the rank} of $W$, to be the minimal number $r$, such that there exist simple tensors $\fromto{s_1}{s_r}$ 
     with $W$ contained in $\langle\fromto{s_1}{s_r}\rangle$.
   For $p \in A_1 \otimes \dotsb \otimes A_d$, we write $R(p):= R(\langle p\rangle)$.
\end{defin}

In the setting of the definition, if $d=1$, then $R(W)= \dim W$.
If $d=2$ and $W = \langle p \rangle$ is $1$-dimensional, then $R(W)$ is the rank of $p$ viewed as a linear map $A_1^*\to A_2$.
If $d=3$ and $W = \langle p \rangle$ is $1$-dimensional, then $R(W)$ is equal to $R(p)$ in the sense of Section~\ref{sect_intro}.
More generally, for arbitrary $d$, one can relate the rank $R(p)$ of $d$-way tensors with the rank $R(W)$ of certain linear subspaces in the space of $(d-1)$-way tensors. 
This relation is based on the \emph{slice technique}, which we are going to review in Section~\ref{section_slices}.

\subsection{Variety of simple tensors}\label{sect_variety_of_simple_tensors}

As it is clear from the definition, the rank of a tensor does not depend on the non-zero rescalings of $p$. 
Thus it is natural and customary to consider the rank as a function on the projective space $\PP(A_1 \otimes A_2 \otimes \dotsb \otimes A_d)$.
There the set of simple tensors is naturally isomorphic to the cartesian product of projective spaces.
Its embedding in the tensor space is also called the \emph{Segre variety}: 
\[
  \Seg=\Seg_{A_1, A_2, \dotsc, A_d} := \PP A_1 \times \PP A_2 \times \dotsb \times \PP A_d \subset \PP(A_1 \otimes A_2 \otimes \dotsb \otimes A_d).
\]

We will intersect linear subspaces of the tensor space with the Segre variety.
Using the language of algebraic geometry, such intersection may have a non-trivial scheme structure.
In this article we just ignore the scheme structure and all  our intersections are set theoretic.
%Not to confuse algebraic geometers,
To avoid ambiguity of notation,
we write $\reduced{(\cdot)}$ to underline this issue, while
the reader not originating from algebraic geometry should ignore the symbol $\reduced{(\cdot)}$.

\begin{notation}\label{notation_V_Seg}
  Given a linear subspace of a tensor space, $V \subset A_1 \otimes A_2 \otimes \dotsb \otimes A_d$,
    we denote:
\[
  V_{\Seg}:=\reduced{(\PP V  \cap \Seg)}. 
\]
Thus $V_{\Seg}$ is (up to projectivisation) 
   the set of rank one tensors in $V$.
\end{notation}

In this setting, we have the following trivial rephrasing of the definition of rank:
\begin{prop}\label{prop_can_choose_decomposition_containing_simple_tensors_from_W}
   Suppose $W \subset A_1 \otimes A_2 \otimes \dotsb \otimes A_d$ is a linear subspace.
   Then $R(W)$ is equal to the minimal number $r$, such that there exists a linear subspace $V \subset A_1 \otimes A_2 \otimes \dotsb \otimes A_d$ of dimension $r$
     with $W \subset V$ and $\PP V $ is linearly spanned by 
     $V_{\Seg}$.
   In particular,
   \begin{enumerate}
      \item $R(W) = \dim W$ if and only if  
            \[
               \PP  W = \langle W_{\Seg} \rangle.
            \]     
      \item Let $U$ be the linear subspace such that $\PP U :=\langle W_{\Seg} \rangle$. 
            Then $\dim U$ tensors from $W$ can be used in the minimal decomposition of $W$, 
              that is, there exist $\fromto{s_1}{s_{\dim U}}\in W_{\Seg}$ such that $W\subset \langle \fromto{s_1}{s_{R(W)}} \rangle$ and $s_i$ are simple tensors.
   \end{enumerate}
\end{prop}

\subsection{Secant varieties and border rank}

For this subsection (and also in Section~\ref{sect_border_rank}) we assume $\kk=\CC$.
 See Remark~\ref{rem_secants_other_fields} for generalisations.
   
In general, the set of tensors of rank at most $r$ is neither open nor closed.
One of the very few exceptions is the case of matrices, that is, tensors in $A\otimes B$.
Instead, one defines the secant variety 
$\sigma_r(\Seg_{\fromto{A_1}{A_d}}) \subset \PP(A_1\otimes \dotsb \otimes A_d)$ as:
\[
   \sigma_r = \sigma_r(\Seg_{\fromto{A_1}{A_d}}) := \overline{\set{p \in \PP(A_1\otimes \dotsb \otimes A_d) \mid  R(p) \le r}}.
\]
The overline $\overline{\set{\cdot}}$ denotes the closure in the Zariski topology. However
  in this definition, the resulting set coincides with the Euclidean closure.
This is a classically studied algebraic variety
    \cite{adlandsvik_joins_and_higher_secant_varieties}, \cite{palatini_secant_varieties}, \cite{zak_tangents},        
   and leads to a definition of border rank of a point.
\begin{defin}\label{def_br_point}
   For $p \in A_1 \otimes A_2 \otimes \dotsb \otimes A_d$
      define $\br(p)$, \emph{the border rank} of $p$,
      to be the minimal number $r$, such that $\linspan{p} \in \sigma_r(\Seg_{\fromto{A_1}{A_d}})$,
      where $\linspan{p}$ is the underlying point of $p$ in the projective space.
   We follow the standard convention that $\br(p) = 0$ if and only if $p=0$.
\end{defin}
Analogously we can give the same definitions for linear subspaces.
Fix $\fromto{A_1}{A_d}$ and an integer $k$. 
Denote by $\Gr(k,A_1\otimes \dotsb \otimes A_d)$ 
   the Grassmannian of $k$-dimensional linear subspaces of the vector space $A_1\otimes \dotsb \otimes A_d$.
Let $\sigma_{r,k}(\Seg) \subset \Gr(k,A_1\otimes \dotsb \otimes A_d)$ be the \emph{Grassmann secant variety}
   \cite{landsberg_jabu_ranks_of_tensors},
   \cite{chiantini_coppens_Grassmannians_of_secant_varieties}, 
   \cite{ciliberto_cools_Grassmann_secant_extremal_varieties}:
\[
   \sigma_{r,k}(\Seg):=\overline{\set{W \in \Gr(k,A_1\otimes \dotsb \otimes A_d) \mid R(W) \le r}}.
\]
\begin{defin}\label{def_br_linear_space}
   For $W \subset A_1 \otimes A_2 \otimes \dotsb \otimes A_d$, a linear subspace of dimension $k$,  
      define $\br(W)$, \emph{the border rank} of $W$,
      to be the minimal number $r$, such that $W \in \sigma_{r,k}(\Seg_{\fromto{A_1}{A_d}})$.
\end{defin}

In particular, if $k =1$, then Definition~\ref{def_br_linear_space} coincides with Definition~\ref{def_br_point}:
   $\br(p) = \br (\linspan{p})$.
An important consequence of the definitions of border rank of a point or of a linear space is that  it is a semicontinuous function 
\[
   \br \colon \Gr(k,A_1\otimes \dotsb \otimes A_d)\to \NN 
\]
for every $k$. Moreover, $\br(p)= 1$ if and only if $\linspan{p}\in Seg$.

\begin{rem}\label{rem_secants_other_fields}
   When treating the border rank and secant varieties  we assume the base field is $\kk=\CC$.
   However, the results of \cite[\S 6, Prop.~6.11]{jabu_jeliesiejew_finite_schemes_and_secants} 
     imply (roughly) that anything that we can say about a secant variety over $\CC$,
     we can also say about the same secant variety over any field $\kk$ of characteristic~$0$. 
   In particular, the same results for border rank over an algebraically closed field $\kk$ will be true.
   If $\kk$ is not algebraically closed, then the definition of border rank above might not generalise 
      immediately, as there might be a difference between the closure in the Zariski topology 
      or in some other topology, the latter being  the Euclidean topology in the case $\kk=\RR$. 
   
\end{rem}

\subsection{Independence of the rank of the ambient space}

As defined above, the notions of rank and border rank of a vector subspace  $W \subset A_1\otimes A_2 \otimes\dotsb \otimes A_d$,
   or of a tensor $p \in A_1 \otimes\dotsb \otimes A_d$, might seem to  depend on the ambient spaces $A_i$.
However, it is well known that the rank is actually independent of the choice of the vector spaces.
We first recall this result for tensors, then we apply the slice technique to show it in general.
\begin{lem}[{\cite[Prop.~3.1.3.1]{landsberg_tensorbook} and \cite[Cor.~2.2]{landsberg_jabu_ranks_of_tensors}}]\label{lem_rank_independent_of_ambient}
   Suppose $\kk=\CC$ and $p \in A_1' \otimes A_2' \otimes \dotsb \otimes A_d'$ for some linear subspaces $A_i'\subset A_i$.
   Then $R(p)$ (respectively, $\br(p)$)
      measured as the rank (respectively, the border rank) in $A_1' \otimes \dotsb \otimes A_d'$ is equal
      to the rank (respectively, the border rank) measured in $A_1 \otimes \dotsb \otimes A_d$.
\end{lem}
We also state a stronger fact about the rank from the same references:
  in the notation of Lemma~\ref{lem_rank_independent_of_ambient}, any minimal expression $W\subset \langle \fromto{s_1}{s_{R(W)}}\rangle$, 
  for simple tensors $s_i$, must be contained in  $A_1' \otimes \dotsb \otimes A_d'$.
Here we show that the difference in the length of the decompositions 
  must be at least the difference of the respective dimensions.
For simplicity of notation, we restrict the presentation to the case $d=3$. 
The reader will easily generalise the argument to any other numbers of factors.
We stress that the lemma below does not depend on the base field, in particular, it does not require algebraic closedness.

\begin{lem}\label{lem_bound_on_rank_for_non_concise_decompositions}
   Suppose that $p \in A' \otimes B \otimes C$, for a linear subspace $A'\subset A$, 
      and that we have an expression $p \in \langle \fromto{s_1}{s_{r}}\rangle$, where $s_i = a_i \otimes b_i \otimes c_i$ are simple tensors.
  Then:
  \[
    r \ge R(p) + \dim \langle \fromto{a_1}{a_r}\rangle - \dim A'.
  \]
\end{lem}
In particular,
Lemma~\ref{lem_bound_on_rank_for_non_concise_decompositions}
implies the rank part of 
  Lemma~\ref{lem_rank_independent_of_ambient} 
  for any base field $\kk$,
  which on its own can also be seen by following 
  the proof
  of \cite[Prop.~3.1.3.1]{landsberg_tensorbook} 
  or \cite[Cor.~2.2]{landsberg_jabu_ranks_of_tensors}.
\begin{prf}
  For simplicity of notation, we assume that $A' \subset \langle \fromto{a_1}{a_r}\rangle$ 
      (by replacing $A'$ with a smaller subspace if needed) and that $A = \langle \fromto{a_1}{a_r}\rangle$
      (by replacing $A$ with a smaller subspace). 
   Set $k = \dim A -\dim A'$ and let us reorder the simple tensors $s_i$ 
      in such a way that the first $k$ of the $a_i$'s are linearly independent 
      and $\linspan{A' \sqcup \setfromto{a_1}{a_k}} = A$.

   Let $A'' = \linspan{\fromto{a_1}{a_k}}$ so that $A = A' \oplus A''$
     and consider the quotient map $\pi\colon A \to A/A''$.
   Then the composition $A' \to A \stackrel{\pi}{\to} A/A''  \simeq A'$ is an isomorphism, denoted by $\phi$.
   By a minor abuse of notation, let $\pi$ and $\phi$ also denote the induced maps 
      $\pi\colon A\otimes B\otimes C \to (A/A'')\otimes B \otimes C$ and 
      $\phi\colon A'\otimes B\otimes C \simeq A'\otimes B \otimes C$.
   We have
   \begin{align*}
       \phi(p) =\pi(p) & \in \pi\left(\linspan{ \fromto{a_1\otimes b_1\otimes c_1}{a_r\otimes b_r\otimes c_r} }\right)\\
                       & = \linspan{ \fromto{\pi(a_1)\otimes b_1\otimes c_1}{\pi(a_r)\otimes b_r\otimes c_r} }\\
                       & = \linspan{ \fromto{\pi(a_{k+1})\otimes b_{k+1}\otimes c_{k+1}}{\pi(a_r)\otimes b_r\otimes c_r} }.
   \end{align*}
   Using the inverse of the isomorphism $\phi$, 
      we get a presentation of $p$ as a linear combination of $(r-k)$ simple tensors, 
      that is, $R(p)\le r-k$ as claimed.
\end{prf}

\subsection{Slice technique and conciseness}\label{section_slices}

We define the notion of conciseness of tensors and we review 
   a standard \emph{slice technique} that replaces the calculation of rank of three way tensors with the calculation of rank of linear spaces of matrices.

A tensor $p \in A_1 \otimes A_2 \otimes \dotsb \otimes A_d$ determines a linear map $p \colon A_1^* \to A_2 \otimes \dotsb \otimes A_d$. 
Consider the image $W = p(A_1^*) \subset A_2 \otimes\dotsb  \otimes  A_d$.
The elements of a basis of $W$ (or the image of a basis of $A_1^*$) are called \emph{slices} of $p$.
The point is that $W$ essentially uniquely (up to an action of $GL(A_1)$) determines $p$ (cfr. \cite[Cor.~3.6]{landsberg_jabu_ranks_of_tensors}).
Thus the subspace $W$ captures the geometric information about $p$, in particular its rank and border rank.

\begin{lem}[{\cite[Thm~2.5]{landsberg_jabu_ranks_of_tensors}}]\label{lem_rank_of_space_equal_to_rank_of_tensor}
   Suppose $p \in A_1 \otimes A_2 \otimes \dotsb \otimes A_d$ and $W = p(A_1^*)$ as above. 
   Then $R(p) = R(W)$ and (if $\kk=\CC$) $\br(p) = \br(W)$.
\end{lem}

Clearly, we can also replace  $A_1$ with any of the $A_i$ to define slices as images $p(A_i^*)$ and obtain the analogue of the lemma. 

We can now prove the analogue of Lemmas~\ref{lem_rank_independent_of_ambient} and \ref{lem_bound_on_rank_for_non_concise_decompositions}
  for higher dimensional subspaces of the tensor space. As before, to simplify the notation,
  we only consider the case $d=2$, which is our main interest.
  
\begin{prop}\label{prop_bound_on_rank_for_non_concise_decompositions_for_vector_spaces}
   Suppose $W \subset B' \otimes C'$ for some linear subspaces $B'\subset B$, $C' \subset C$. 
\begin{enumerate}
   \item \label{item_rank_can_be_measured_anywhere}
      The numbers $R(W)$ and $\br(W)$ measured as the rank and border rank of $W$ in $B' \otimes C'$
        are equal to its rank and border rank calculated in $B \otimes C$ (in the statement about border rank, we assume that $\kk=\CC$).
   \item \label{item_decompositions_in_larger_spaces}
      Moreover, if we have an expression $W \subset \langle \fromto{s_1}{s_{r}}\rangle$, 
        where $s_i = b_i \otimes c_i$ are simple tensors,
        then:
  \[
    r \ge R(W) + \dim \langle \fromto{b_1}{b_r}\rangle - \dim B'
  \]
\end{enumerate}

\end{prop}

\begin{prf}
  Reduce to Lemmas~\ref{lem_rank_independent_of_ambient} and \ref{lem_bound_on_rank_for_non_concise_decompositions} 
    using Lemma~\ref{lem_rank_of_space_equal_to_rank_of_tensor}.
\end{prf}

We conclude this section by recalling the following definition.
\begin{defin}
   Let $p \in A_1 \otimes A_2 \otimes \dotsb \otimes A_d$ be a tensor or let $W \subset A_1 \otimes A_2 \otimes \dotsb \otimes A_d$ be a linear subspace. 
   We say that $p$ or $W$ is \emph{$A_1$-concise} if for all linear subspaces $V \subset A_1$, if $p \in V \otimes A_2 \otimes \dotsb \otimes A_d$
      (respectively,  $W \subset V \otimes A_2 \otimes \dotsb \otimes A_d$), then $V = A_1$.
   Analogously, we define $A_i$-concise tensors and spaces for $i =\fromto{2}{d}$. 
   We say $p$ or $W$ is \emph{concise} if it is $A_i$-concise for all $i\in \setfromto{1}{n}$.
\end{defin}

\begin{rem}
Notice that $p \in A_1 \otimes A_2 \otimes \dotsb \otimes A_d$ is $A_1$-concise if and only if $p\colon A_1^* \to A_2 \otimes \dotsb \otimes A_d$ is injective.
\end{rem}

\section{Direct sum tensors and spaces of matrices}\label{sect_dir_sums}

Again, for simplicity of notation we restrict the presentation to the case of tensors in $A\otimes B \otimes C$ or linear subspaces of $B \otimes C$.

We introduce the following notation that will be adopted throughout this manuscript.

\begin{notation}\label{notation}
Let $A',A'',B',B'',C',C''$ be vector spaces over $\kk$ of dimensions,  respectively, $\bfa',\bfa'',\bfb',\bfb'',\bfc',\bfc''$.
Suppose $A = A' \oplus A''$,  $B = B' \oplus B''$, $C = C' \oplus C''$ 
   and $\bfa=\dim A = \bfa'+\bfa''$, $\bfb=\dim B =\bfb'+\bfb''$ and $\bfc=\dim C =\bfc'+\bfc''$.

For the purpose of illustration, we will interpret the two-way tensors in $B \otimes C$ 
as matrices in $\ccM^{\bfb\times \bfc}$. 
This requires choosing bases of $B$ and $C$, but (whenever possible) we will refrain from naming the bases explicitly.
We will refer to an element of the space of matrices 
$\ccM^{\bfb\times \bfc} \simeq B\ts C$ as a $(\bfb'+\bfb'',\bfc'+\bfc'')$ \emph{partitioned matrix}.
Every matrix $w\in\ccM^{\bfb\times \bfc} $ is a block matrix with four blocks of size
$\bfb'\times \bfc'$, $\bfb'\times \bfc''$, 
$\bfb''\times \bfc'$ and $\bfb''\times \bfc''$ respectively.
\end{notation}

\begin{notation}\label{notation_2}
As in Section~\ref{section_slices}, a tensor $p\in A\otimes B\otimes C$ is a linear map $p:A^\ast\to B\otimes C$;
   we denote by $W:=p(A^\ast)$ the image of $A^\ast$ in the space of matrices $B\otimes C$.
   Similarly, if $p = p' + p'' \in (A' \oplus A'') \otimes (B' \oplus B'') \otimes (C' \oplus C'')$ is such that  $p'\in A'\otimes  B'\otimes C'$ and $p''\in A''\otimes  B''\otimes C''$, we set $W':=p'({A'}^\ast)\subset B'\otimes C'$ and $W'':=p''({A''}^\ast)\subset B''\otimes C''$. 
%    We will frequently omit subscript, if there is no risk of confusion, using just $W,W',W''$.
In such situation, we will say that $p = p' \oplus p''$ is a \emph{direct sum tensor}.

We have the following direct sum decomposition:
\[
  W=W'\oplus W''\subset (B'\otimes C')\oplus (B''\otimes C'')
\]
and an induced matrix partition of type $(\bfb'+\bfb'',\bfc'+\bfc'')$ on every matrix $w\in W$  such that 
\[
  w=\begin{pmatrix} 
       w'           & \underline{0} \\
      \underline{0} & w''
    \end{pmatrix},
\]
where $w'\in W'$ and $w''\in W''$, and the two $\underline{0}$'s denote zero matrices of size
$\bfb'\times\bfc''$ and $\bfb''\times\bfc'$ respectively.
\end{notation}

\begin{prop}
Suppose that $p$, $W$, etc.~are as in Notation~\ref{notation_2}.
Then the additivity of the rank holds for $p$, that is~$R(p) = R(p')+R(p'')$, if and only if the additivity of the rank holds for $W$, that is, $R(W) = R(W')+ R(W'')$.
\end{prop}
\begin{prf}
It is an immediate consequence of Lemma~\ref{lem_rank_of_space_equal_to_rank_of_tensor}.
\end{prf}

\subsection{Projections and decompositions} 

The situation we consider here again concerns the direct sums and their minimal decompositions.
We fix $W' \subset B'\otimes C'$ and $W'' \subset B''\otimes C''$ and we choose a minimal decomposition of $W'\oplus W''$, 
that is, a linear subspace $V\subset B \otimes C$
  such that $\dim V = R(W'\oplus W'')$, $\PP V=\linspan{V_{\Seg}}$ and $V\supset W'\oplus W''$.
Such linear spaces $W'$, $W''$ and $V$ will be fixed for the rest of Sections~\ref{sect_dir_sums} and \ref{sec_rank_one_matrices_and_additive_rank}.

In addition to  Notations~\ref{notation_V_Seg}, \ref{notation} and \ref{notation_2} we need the following.

\begin{notation}\label{notation_projection}
Under Notation~\ref{notation},  let $\pi_{C'}$ denote the projection
\[
   \pi_{C'}:C \to C'',\text{ or } 
\]
whose kernel is the space $C'$. With slight abuse of notation, we shall  denote by $\pi_{C'}$ also the following projections
   \[
   \pi_{C'}:B\ts C\to B\ts C'', \text{ or } \pi_{C'}:A\otimes B\ts C\to A \otimes B\ts C'',
\]
with kernels, respectively, $B\otimes C'$ and $A\otimes B \otimes C'$.
The target of the projection is regarded as a subspace of $C$, $B\otimes C$, or $A\otimes B \otimes C$, so that it is possible to compose such projections, for instance:
    \[
        \pi_{C'} \pi_{B''} \colon B\otimes C \to B'\otimes C'', \text{ or } \pi_{C'} \pi_{B''} \colon A \otimes  B\otimes C \to A\otimes B'\otimes C''.
    \]
We also let $E'\subset B' $ (resp. $E''\subset B''$) be the minimal vector subspace such that 
   $\pi_{C'}(V)$ (resp. $\pi_{C''}(V)$) is contained in $(E'\oplus B'')\ts C''$ (resp. $(B'\oplus E'')\ts C'$).

By swapping the roles of $B$ and $C$, we  define $F'\subset C'$ and $F''\subset C''$ analogously. By the lowercase letters $\bfe',\bfe'',\bff',\bff''$ we denote the dimensions of the subspaces $E',E'',F',F''$.
\end{notation}

If the differences $R(W') - \dim W'$ and $R(W'') - \dim W''$ (which we will informally call the \emph{gaps}) are large, then the spaces $E',E'',F',F''$ could be large too, in particular they can coincide with $B', B'', C', C''$ respectively. 
In fact, these spaces measure ``how far'' a minimal decomposition $V$ of a direct sum $W=W' \oplus W''$ is from being a direct sum of decompositions of $W'$ and $W''$. 

In particular, we will show in 
  Proposition~\ref{prop_SAC_if_E'=0} and Corollary~\ref{cor_small_dimensions_of_E_and_F},
  that if $E'' = \set{0}$ or if both $E''$ and $F''$ are sufficiently small, then $R(W) = R(W')+ R(W'')$. 
Then, as a consequence of Corollary~\ref{cor_bounds_on_es_and_fs},
  if one of the gaps is at most two (say, $R(W'') = \dim W'' +2$), 
  then the additivity of the rank holds, see Theorem~\ref{thm_additivity_rank_plus_2}.

\begin{lem}\label{lemma_bound_r'_e'_R_w'}
In Notation~\ref{notation_projection} as above, with $W=W'\oplus W'' \subset B\otimes C$, 
the following inequalities hold.
\begin{align*}
R(W') + \bfe'' & \le R(W)-\dim W'',&
R(W'')+ \bfe'  & \le R(W)-\dim W', \\
R(W') + \bff'' & \le R(W)-\dim W'',&
R(W'')+ \bff'  & \le R(W)-\dim W'.
\end{align*}
\end{lem}
\begin{proof}
We prove only the first inequality $R(W') + \bfe'' \le R(W)-\dim W''$, 
  the other follow in the same way by swapping $B$ and $C$ or ${}'$ and ${}''$.
By Proposition~\ref{prop_bound_on_rank_for_non_concise_decompositions_for_vector_spaces}\ref{item_rank_can_be_measured_anywhere} and \ref{item_decompositions_in_larger_spaces}
we may assume $W'$ is concise: $R(W')$ or $R(W)$ are not affected by choosing the minimal subspace of $B'$ by \ref{item_rank_can_be_measured_anywhere}, also the minimal decomposition $V$ cannot involve anyone from outside of the minimal subspace by \ref{item_decompositions_in_larger_spaces}.

Since $V$ is spanned by rank one matrices and the projection
  $\pi_{C''}$ preserves the set of matrices of rank at most one,
   also the vector space $\pi_{C''}(V)$ is 
  spanned by rank one matrices, say 
  \[
    \pi_{C''}(V) = \linspan{\fromto{b_1 \otimes c_1}{b_r\otimes c_r}}
  \]
  with $r= \dim \pi_{C''}(V)$.
Moreover, $\pi_{C''}(V)$ contains $W'$.
We claim that 
\[
   B'\oplus E'' = \linspan{\fromto{b_1}{b_r}}.
\]
Indeed, the inclusion  $B' \subset  \linspan{\fromto{b_1}{b_r}}$
  follows from the conciseness of $W'$,
  as $W' \subset V\cap B'\otimes C'$.
Moreover, the inclusions
$E''\subset \linspan{\fromto{b_1}{b_r}}$ and   $B'\oplus E'' \supset \linspan{\fromto{b_1}{b_r}}$
follow from the definition of $E''$, cf. Notation~\ref{notation_projection}.

Thus Proposition~\ref{prop_bound_on_rank_for_non_concise_decompositions_for_vector_spaces}\ref{item_decompositions_in_larger_spaces} 
  implies that 
  \begin{equation}\label{equ_bound_on_pi_C_bis_of_V}
    r= \dim\pi_{C''}(V) \ge  R(W')+\underbrace{\dim\linspan{\fromto{b_1}{b_r}}}_{\bfb'+\bfe''}- \bfb' = R(W') +\bfe''.
  \end{equation}
Since $V$ contains $W''$ and $\pi_{C''}(W'') = \set{0}$, we have
\[
  r= \dim \pi_{C''}(V) \le \dim V - \dim W'' = R(W) - \dim W''. 
\]
The claim follows from the above inequality together with  \eqref{equ_bound_on_pi_C_bis_of_V}.
\end{proof}

Rephrasing the inequalities of Lemma~\ref{lemma_bound_r'_e'_R_w'}, we obtain the following.
\begin{cor}\label{cor_bounds_on_es_and_fs}
   If $R(W) < R(W')+R(W'')$, then 
   \begin{align*}
       \bfe' &< R(W')  - \dim W', &
       \bff' &< R(W')  - \dim W', \\
       \bfe''&< R(W'') - \dim W'',&
       \bff''&< R(W'') - \dim W''.
   \end{align*}
\end{cor}
This immediately recovers the known case of additivity, when the gap is equal to $0$,
that is, if $R(W')=\dim W'$, then $R(W)=R(W')+R(W'')$ (because $\bfe'\ge 0$).
Moreover, it implies that if one of the gaps is equal to $1$ 
   (say  $R(W')=\dim W'+1$),
   then either the additivity holds or both $E'$ and $F'$ 
   are trivial vector spaces.
In fact, the latter case is only possible if the former case 
   holds too.
\begin{lem}\label{lem_rank_at_least_2_more_than_dimension}
   With Notation~\ref{notation_projection}, suppose $E'=\set{0}$ and $F'=\set{0}$. 
   Then the additivity of the rank holds $R(W)= R(W') + R(W'')$.
   In particular, if $R(W') \le \dim W' +1$, then the additivity holds.
\end{lem}
\begin{proof}
   Since $E'=\set{0}$ and $F'=\set{0}$, by the definition of $E'$ and $F'$ we must have the following inclusions:
   \[
     \pi_{B''}(V) \subset B'\otimes C' \text{ and } \pi_{C''}(V) \subset B'\otimes C'.
   \]
   Therefore $V \subset  B'\otimes C' \oplus B''\otimes C''$ and $V$ is obtained from the union of the decompositions of $W'$ and $W''$.
   
   The last statement follows from Corollary~\ref{cor_bounds_on_es_and_fs}
\end{proof}

Later in Proposition~\ref{prop_SAC_if_E'=0} we will show a stronger version of the above lemma, 
 namely  that it is sufficient to assume that only one of $E'$ or $F'$ is zero.
In Corollary~\ref{cor_small_dimensions_of_E_and_F} 
   we prove a further generalisation based on the results in the following subsection.

\subsection{``Hook''-shaped spaces}\label{sec_hook_shaped_spaces}

It is known since \cite{jaja_takche_Strassen_conjecture} that the additivity of the tensor rank holds for tensors with one of the factors of dimension $2$, that is, using Notation~\ref{notation} and \ref{notation_2}, 
  if $\bfa' \le 2$ then $R(p'+p'') = R(p')+R(p'')$.
The same claim is recalled in \cite[Sect.~4]{landsberg_michalek_abelian_tensors} after Theorem~4.1. 
The brief comment says that if rank of $p'$ can be calculated by the \emph{substitution method}, then the additivity of the rank holds. 
Landsberg and Micha{\l}ek implicitly suggest that if $\bfa' \le 2$, then the rank of $p'$ can be calculated by the substitution method, \cite[Items~(1)--(6) after Prop.~3.1]{landsberg_michalek_abelian_tensors}.
This is indeed the case (at least over an algebraically closed field $\kk$), although rather demanding to verify,
   at least in the version of the algorithm presented in the cited article.
In particular, to show that the substitution method can calculate the rank of $p'\in \kk^2\otimes B'\otimes C'$, 
   one needs to use the normal forms of such tensors \cite[\S10.3]{landsberg_tensorbook} 
   and understand all the cases, and it is hard to agree that this method is so much simplier than the original approach of \cite{jaja_takche_Strassen_conjecture}.

Instead, probably, the intention of the authors of \cite{landsberg_michalek_abelian_tensors} was slightly different, 
   with a more direct application of \cite[Prop.~3.1]{landsberg_michalek_abelian_tensors} 
   (or Proposition~\ref{proposition_for_AFT_method_coordinate_free} below).
This has been carefully detailed and described in \cite[Prop.~3.2.12]{rupniewski_mgr} and  
   here we present this approach to show a stronger statement about small ``hook''-shaped spaces (Proposition~\ref{prop_1_2_hook_shaped}).
We stress that our argument 
  for Proposition~\ref{prop_1_2_hook_shaped}, 
  as well as \cite[Prop.~3.2.12]{rupniewski_mgr} requires the assumption of an algebraically closed base field $\kk$, 
  while the original approach 
  of \cite{jaja_takche_Strassen_conjecture} 
  works over any field.
  For a short while we also work over an arbitrary field.

\begin{defin}\label{def_hook_shaped_space}
   For non-negative integers $e, f$, we say that a linear subspace $W \subset B \otimes C$ 
      is \emph{$(e,f)$-hook shaped}, 
      if $W\subset \kk^e\otimes C + B \otimes \kk^f$ 
      for some choices of linear subspaces $\kk^e\subset B$ and $\kk^f\subset C$.
\end{defin}

The name ``hook shaped'' space comes from the fact that under an appropriate choice of basis, the only non-zero coordinates form a shape of a hook 
$\ulcorner$ situated in the upper left corner of the matrix, see Example~\ref{ex_hook}.
The integers $(e,f)$ specify how wide the edges of the hook are.
A similar name also appears in the context of Young diagrams, see for instance \cite[Def.~2.3]{berele_regev_hook_Young_diagrams_with_applications}.

\begin{example}\label{ex_hook}
A $(1,2)$-hook shaped subspace of $\kk^4 \otimes \kk^4$ has only the following possibly nonzero entries in some coordinates:
$$\begin{bmatrix}
  *   & * & * & *\\
  * & * & 0 & 0\\
  * & *  & 0 & 0\\
  * & *  & 0 & 0\\
 \end{bmatrix}.
$$
\end{example}

The following elementary observation is presented in \cite[ Prop.~3.1]{landsberg_michalek_abelian_tensors} and in \cite[Lem.~B.1]{alexeev_forbes_tsimerman_Tensor_rank_some_lower_and_upper_bounds}.
Here we have phrased it in a coordinate free way.

% \begin{prop}[{\cite[Prop.~3.1]{landsberg_michalek_abelian_tensors}}] \label{proposition_for_AFT_method} 
%    Let $p \in A\otimes B \otimes C$, $R(p) = r$, $\{c_i \}_{i=1,\dotsc,\bfc}$ be a fixed basis of $C$,
%    $p = \sum_{i=1}^\bfc M_i \otimes c_i$, where $ M_i \in A \otimes B$ and $M_{\bfc} \neq 0$.
%  For constants $\lambda_1, \dotsc, \lambda_{\bfc-1} \in \kk$ (jointly denoted by $\lambda\in \kk^{\bfc-1}$) we define 
%  \begin{equation}\label{tensor ze stwierdzenia do AFT}
%     \tilde{p}= \tilde{p}_{\lambda}:= \sum_{i=1}^{\bfc-1} (M_i - \lambda_i M_{\bfc}) \otimes c_i \in A \otimes B 
%                                     \otimes \linspan{\fromto{c_1}{c_{\bfc-1}}}.
% \end{equation}
%  Then:
%  \begin{enumerate}
%   \item there exists a choice of constants $\lambda$, such that $R(\tilde{p}_{\lambda}) \leq r-1$, and
%   \item if in addition $R(M_{\bfc})=1$, then for any choice of $\lambda$ we have $R(\tilde{p}_{\lambda}) \geq r-1$.
%  \end{enumerate}
% \end{prop}
% 

\begin{prop} \label{proposition_for_AFT_method_coordinate_free} 
   Let $p \in A\otimes B \otimes C$, $R(p) = r>0$, and pick $\alpha \in A^*$ such that $p(\alpha) \in B\otimes C$ is nonzero.
   Consider two hyperplanes in $A$: the linear hyperplane $\alpha^{\perp}= (\alpha =0)$
     and the affine hyperplane $(\alpha=1)$.
   For any $a\in (\alpha=1)$, denote
   \[
     \tilde{p}_{a}:= p - a\otimes p(\alpha) \in \alpha^{\perp} \otimes B \otimes C.
   \]
   Then:
 \begin{enumerate}
  \item \label{item_AFT_exists_choice_droping_rank}
        there exists a choice of $a\in (\alpha=1)$ such that $R(\tilde{p}_{a}) \leq r-1$,
  \item \label{item_AFT_any_choice_drops_rank_at_most_1}
        if in addition $R(p(\alpha))=1$, then for any choice of $a \in (\alpha=1)$ 
        we have $R(\tilde{p}_{a}) \geq r-1$.
 \end{enumerate}
\end{prop}
See \cite[Prop.~3.1]{landsberg_michalek_abelian_tensors} for the proof (note the statement there is over the complex numbers only, but the proof is field independent) or, alternatively,
   using Lemma~\ref{lem_rank_of_space_equal_to_rank_of_tensor} translate it into 
   the following straightforward  statement on linear spaces of tensors:

% In terms of slices, the statement has the following two interpretations, depending on the role of $C$.
% 
% \begin{prop}\label{proposition_for_AFT_method_slice_C_coordinate_free} 
%  Suppose $W \subset  B \otimes C$ is a linear subspace, $R(W) = r$, 
%     and the embedding $p \colon W \to  B\otimes C$.
%  Pick $\gamma\in C^*$ such that $p(\gamma)\colon W \to B$ is a nonzero linear map.
%  For $c\in (\gamma=1)$ define 
%  \begin{equation}\label{tensor ze stwierdzenia do AFT_2_coord_free}
%     \widetilde{W}(c):= \set{p(w) - c\otimes  p(w \otimes \gamma) \mid w \in W}
%        \subset B\otimes \gamma^{\perp}. 
% \end{equation}
%  Then:
%  \begin{enumerate}
%   \item there exists a choice  $c\in (\gamma=1)$, such that $R(\widetilde{W}(c)) \leq r-1$, and
%   \item if in addition $R(p(\gamma)=1$, 
%          then for any choice of  $c\in (\gamma=1)$ we have $R(\widetilde{W}(c)) \geq r-1$.
%  \end{enumerate}
% \end{prop}
% 
% 
\begin{prop}\label{proposition_for_AFT_method_slice_A} 
 Suppose $W \subset  B \otimes C$ is a linear subspace, $R(W) = r$. 
 Assume $w \in W$ is a non-zero element.
 Then:
 \begin{enumerate}
  \item there exists a choice of a complementary subspace $\widetilde{W}\subset W$,
     such that $\widetilde{W} \oplus \linspan{w} = W$ and $R(\widetilde{W}) \leq r-1$, and
  \item if in addition $R(w)=1$, then for any choice of the complementary subspace $\widetilde{W} \oplus \linspan{w} = W$
     we have $R(\widetilde{W}) \geq r-1$.
 \end{enumerate}
\end{prop}
% 
% 
% 
% 
% \begin{prop}\label{proposition_for_AFT_method_slice_C} 
%  Suppose $W \subset  B \otimes C$ is a linear subspace, $R(W) = r$, $\{ c_i \}_{i=1,\dotsc,\bfc}$ be a fixed basis of $C$,
%     and the embedding $p \colon W \to  B\otimes C$ is given by $p(w) = \sum_{i=1}^{\bfc} M_i(w) \otimes c_i$ 
%     for some linear maps $M_i \colon W\to B$. 
%  Assume $M_{\bfc} \ne 0$.
%  For constants $\lambda=(\lambda_1,\dotsc, \lambda_{\bfc-1}) \in \kk^{\bfc-1}$ 
%     we define the following linear subspace of $B\otimes \linspan{\fromto{c_1}{c_{\bfc-1}}}$:
%  \begin{equation}\label{tensor ze stwierdzenia do AFT_2}
%     \widetilde{W}^{\lambda}:= \set{\sum_{i=1}^{\bfc-1} (M_i(w) - \lambda_i M_1(w)) \otimes c_i \mid w \in W}. 
% \end{equation}
%  Then:
%  \begin{enumerate}
%   \item there exists a choice of constants $\lambda$, such that $R(\widetilde{W}^{\lambda}) \leq r-1$, and
%   \item if in addition $R(M_{\bfc})=1$, then for any choice of $\lambda$ we have $R(\widetilde{W}^{\lambda}) \geq r-1$.
%  \end{enumerate}
% \end{prop}

Proposition~\ref{proposition_for_AFT_method_coordinate_free} is crucial in the proof that the additivity of the rank holds for vector spaces, one of which is $(1,2)$-hook shaped (provided that the base field is algebraically closed).
Before taking care of that, we use the same proposition to prove a simpler statement about $(1,1)$-hook shaped spaces, which is valid without any assumption on the field. The proof essentially follows the idea outlined in 
\cite[Thm~4.1]{landsberg_michalek_abelian_tensors}.

\begin{prop}\label{prop_1_1_hook_shaped}
   Suppose $W''\subset B''\otimes C''$ is $(1,1)$-hook shaped and $W'\subset B'\otimes C'$ 
      is an arbitrary subspace.
   Then the additivity of the rank holds for $W'\oplus W''$.
\end{prop}

Before commencing the proof of the proposition we state three lemmas, which will be applied to both $(1,1)$ and $(1,2)$ hook shaped spaces.
The first lemma is analogous to \cite[Thm~4.1]{landsberg_michalek_abelian_tensors}. 
In this lemma (and also in the rest of this section) we will work with a sequence of tensors, $p_0, p_1, p_2, \dotsc$ in the space $A\otimes B \otimes C$, which are not necessarily direct sums.
Nevertheless, for each $i$,
  we write $p'_i = \pi_{A''}  \pi_{B''} \pi_{C''}(p_i)$ 
  (that is, this is the ``corner'' of $p_i$ corresponding to $A'$, $B'$ and $C'$).   
We define $p''_i$ analogously.

\begin{lem}\label{lem_proving_additivity_via_substitution}
   Suppose $W'\subset A'\otimes B'\otimes C'$ 
     and $W''\subset A''\otimes B''\otimes C''$ are two subspaces.
   Let $r''=R(W'')$ and suppose that there exists a sequence of tensors
   $p_0, p_1, p_2, \dotsc, p_{r''}\in A\otimes B \otimes C$ 
     satisfying the following properties:
    \renewcommand{\theenumi}{(\arabic{enumi})}
     \begin{enumerate}
      \item \label{item_proof_additivity_hook_p0_eq_p}
            $p_0 =p$ is such that $p(A^*) =  W = W' \oplus W''$,
      \item \label{item_proof_additivity_hook_p_prime_preserved} 
            $p'_{i+1}=p'_{i}$ for every $0\le i < r''$,
      \item \label{item_proof_additivity_hook_p_i_bis_does_not_drop_rank_too_much} 
            $R(p''_{i+1}) \ge R(p''_{i})-1$ for every $0\le i < r''$, 
      \item \label{item_proof_additivity_hook_p_i_drops_rank_enough} 
            $R(p_{i+1}) \le R(p_{i})-1$ for each $0\le i < r''$.
   \end{enumerate}
   \renewcommand{\theenumi}{(\roman{enumi})}

   Then the additivity of the rank holds for $W'\oplus W''$
      and for each $i < r''$ we must have $p''_i\ne 0$.
\end{lem}
\begin{proof}
   We have 
   \[
      R(W') + R(W'')
      \stackrel{\text{\ref{item_proof_additivity_hook_p0_eq_p},\ref{item_proof_additivity_hook_p_prime_preserved}}}{=}
      R(p'_{r''}) + r'' 
      \le R(p_{r''})+r''
      \stackrel{\text{\ref{item_proof_additivity_hook_p_i_drops_rank_enough}}}{\le}
      R(p_0)
      \stackrel{\text{\ref{item_proof_additivity_hook_p0_eq_p}}}{=}
      R(W).
   \]
   The nonvanishing of $p_i''$ follows from
   \ref{item_proof_additivity_hook_p_i_bis_does_not_drop_rank_too_much}.
\end{proof}

The second lemma tells us how to construct a single step in the above sequence.
\begin{lem}\label{lem_constructing_next_tensor_in_the_sequence}
   Suppose $\Sigma \subset A\otimes B \otimes C$ is a linear subspace, 
     $p_i\in \Sigma$ is a tensor,
     and $\gamma\in C''$ is such that:
     \begin{itemize}
        \item $R(p''_i(\gamma))=1$,
        \item $\gamma$ preserves $\Sigma$, that is,
        $\Sigma(\gamma)\otimes C \subset \Sigma$, 
        where $\Sigma(\gamma) = \set{t(\gamma)\mid t \in \Sigma} \subset A\otimes B$. 
        \item $\Sigma(\gamma)$ does not have entries in $A'\otimes B'$, that is
        $
          \pi_{A''}\pi_{B''}(\Sigma(\gamma)) =0.
        $
     \end{itemize}
   Consider $\gamma^{\perp} \subset C$ to be the perpendicular hyperplane.
   Then there exists 
      $p_{i+1} \in (\Sigma \cap A\otimes B \otimes \gamma^{\perp})$
      that satisfies properties \ref{item_proof_additivity_hook_p_prime_preserved}--\ref{item_proof_additivity_hook_p_i_drops_rank_enough}
      of Lemma~\ref{lem_proving_additivity_via_substitution}
      (for a fixed $i$).
\end{lem}

\begin{proof}
  As in Proposition~\ref{proposition_for_AFT_method_coordinate_free}
    for $c\in (\gamma=1)$ set  $(\tilde{p}_i)_c = p_i - p_i(\gamma)\otimes c  \in A\otimes B\otimes \gamma^{\perp}$.
  We will pick $p_{i+1}$ among the $(\tilde{p}_i)_c$.
  In fact by  
  Proposition~\ref{proposition_for_AFT_method_coordinate_free}\ref{item_AFT_exists_choice_droping_rank}
     there exists a choice of $c$ such that  $p_{i+1}= (\tilde{p}_i)_c$ has rank less than $R(p_i)$, 
     that is, \ref{item_proof_additivity_hook_p_i_drops_rank_enough} is satisfied.
  On the other hand, since $\gamma$ is  in $(C'')^*$, we have $p''_{i+1} = \left(\widetilde{p''_i}\right)_{c''}$ 
    (where $c= c' + c''$ with $c'\in C'$ and $c''\in C''$)
     and by 
     Proposition~\ref{proposition_for_AFT_method_coordinate_free}\ref{item_AFT_any_choice_drops_rank_at_most_1}
     also \ref{item_proof_additivity_hook_p_i_bis_does_not_drop_rank_too_much} is satisfied.
  Property~\ref{item_proof_additivity_hook_p_prime_preserved} follows,
     as $\Sigma(\gamma)$ (in particular, $p_i(\gamma)$) 
     has no entries in $A'\otimes B'\otimes C'$.  
  Finally, $p_{i+1} \in \Sigma$ thanks to the assumption that $\gamma$ preserves $\Sigma$ and $\Sigma$ is a linear subspace.  
\end{proof}

The next lemma is the common first step in the proofs of 
  additivity for $(1,1)$ and $(1,2)$ hooks:  we construct a few initial elements of the sequence needed in 
  Lemma~\ref{lem_proving_additivity_via_substitution}.
  
\begin{lem}\label{lem_first_step_for_hooks}
   Suppose $W''\subset B''\otimes C''$ 
      is a $(1,f)$-hook shaped space for some integer $f$ and $W'\subset B'\otimes C'$ is arbitrary.
   Fix $\kk^1\subset B''$ and $\kk^f\subset C''$ 
      as in Definition~\ref{def_hook_shaped_space} for $W''$.
   Then there exists a sequence of tensors
   $p_0, p_1, p_2, \dotsc, p_{k}\in A\otimes B \otimes C$ for some $k$ that satisfies properties 
   \ref{item_proof_additivity_hook_p0_eq_p}--\ref{item_proof_additivity_hook_p_i_drops_rank_enough} 
   of Lemma~\ref{lem_proving_additivity_via_substitution} and in addition 
   $p''_{k}\in A''\otimes B''\otimes \kk^f$ 
   and for every $i$ 
   we have $ p_i\in A' \otimes B' \otimes C' \oplus 
      A'' \otimes \left(B''\otimes \kk^f  + \kk^1 \otimes C\right)$.
   In particular: 
   \begin{itemize}
    \item $p''_i((A'')^*)$ is a $(1,f)$-hook shaped space
              for every $i<k$,
          while $p''_k((A'')^*)$ is a $(0,f)$-hook shaped space.
    \item Every $p_i$ is ``almost'' a direct sum tensor, that is,
          $
             p_i = (p'_i\oplus p''_i) + q_i,
          $
          where
          \[
            q_i\in A''\otimes \kk^1 \otimes C'\subset A''\otimes B'' \otimes C'.
          \]
   \end{itemize}
\end{lem}

\begin{proof}
   To construct the sequence $p_i$ we recursively apply 
      Lemma~\ref{lem_constructing_next_tensor_in_the_sequence}.
   By our assumptions, 
   $p''\in A'' \otimes B''\otimes \kk^f 
      + A''\otimes \linspan{x}\otimes C''$
   for some choice of $x \in B''$ and fixed $\kk^f \subset C''$.
   We let 
   $\Sigma = A' \otimes B' \otimes C' \oplus A'' \otimes \left(B''\otimes \kk^f  + \linspan{x} \otimes C\right)$.
 
   Tensor $p_0$ is defined by \ref{item_proof_additivity_hook_p0_eq_p}.
   Suppose we have already constructed 
   $p_0,\dotsc, p_i$ and that $p''_i$ is not yet contained in 
     $A'' \otimes B''\otimes \kk^f$.
   Therefore  
     there exists a hyperplane $\gamma^{\perp}=(\gamma=0) \subset C$ 
     for some $\gamma\in (C'')^*\subset C^*$
     such that $\kk^f \subset \gamma^{\perp}$,
     but $p''_{i} \notin A'' \otimes  B''\otimes \gamma^{\perp} $.
  Equivalently, $p''_{i}(\gamma) \ne 0$ 
    and $p''_{i}(\gamma) \subset A''\otimes \linspan{x}$.
  In particular, $R(p''_{i}(\gamma))=1$ 
    and $\Sigma(\gamma) \subset A'' \otimes \linspan{x}$.
  Thus $\gamma$ preserves $\Sigma$ as in 
    Lemma~\ref{lem_constructing_next_tensor_in_the_sequence}
  and $\Sigma(\gamma)$ has no entries in $A'\otimes B'\otimes C'$.

  Thus we construct $p_{i+1}$ using 
    Lemma~\ref{lem_constructing_next_tensor_in_the_sequence}.
  Since we are gradually reducing the dimension of the third factor 
     of the tensor space containing 
     $p''_{i+1}$, eventually we will arrive at the case
     $p''_{i+1} \in A'' \otimes B''\otimes \kk^f$, proving the claim.
\end{proof}

\begin{proof}[Proof of Proposition~\ref{prop_1_1_hook_shaped}]
   We construct the sequence $p_i$ 
     as in Lemma~\ref{lem_proving_additivity_via_substitution}.
   The initial elements $\fromto{p_0}{p_k}$ of the sequence are given by Lemma~\ref{lem_first_step_for_hooks}.
   By the lemma and our assumptions, 
   $p_i''\in A'' \otimes B''\otimes \linspan{y} 
      + A''\otimes \linspan{x}\otimes C''$
   for some choices of $x \in B''$ and $y \in C''$ and 
   \[
       p_k \in A' \otimes B' \otimes C' 
           \oplus A''\otimes \big(\linspan{x} \otimes C' \oplus
           B'' \otimes \linspan{y}\big).
   \]
   Now suppose that we have constructed 
   $\fromto{p_k}{p_j}$ for some $j\ge k$ satisfying \ref{item_proof_additivity_hook_p_prime_preserved}--\ref{item_proof_additivity_hook_p_i_drops_rank_enough},
   such that
   \begin{equation*}%\label{equ_proof_of_1_1_hook_inductive_property_second_step}
    p_j \in  \Sigma = A' \otimes B' \otimes C' 
           \oplus A''\otimes B \otimes (C'\oplus \linspan{y}).
   \end{equation*}
  If $p''_j=0$, 
     then by Lemma~\ref{lem_proving_additivity_via_substitution} we are done, as $j=r''$.
  So suppose $p''_j\ne 0$,
      and choose $\beta \in (B'')^*$ such that $p''_j(\beta) \ne 0$, 
     that is, $R(p''_j(\beta))=1$ since $p''_j(\beta) \in A'' \otimes \linspan{y}$.
  We produce $p_{j+1}$ using 
  Lemma~\ref{lem_constructing_next_tensor_in_the_sequence} 
     with the roles of $B$ and $C$ swapped
     (so also $\beta$ takes the role of $\gamma$ etc.).
     
  We stop after constructing $p_{r''}$ and thus the desired sequence exists and proves the claim.
\end{proof}

In the rest of this section we will show that an analogous statement 
  holds for $(1,2)$-hook shaped spaces under an additional assumption
  that the base field is algebraically closed.
We need the following lemma (false for nonclosed fields),
  whose proof is  a straightforward dimension count, 
  see also \cite[Prop.~3.2.11]{rupniewski_mgr}.

\begin{lem}\label{lem_dim_2_have_rank_one_matrix}
    Suppose $\kk$ is algebraically closed (of any characteristic) 
       and $p\in A\otimes B \otimes \kk^2$ and $p\ne 0$.
    Then at least one of the following holds:
    \begin{itemize}
       \item there exists a rank one matrix in $p(A^*)\subset B\otimes \kk^2$, or
       \item for any $x\in B$ there exists a rank one matrix in $p(x^{\perp}) \subset A \otimes \kk^2$, where $x^{\perp} \subset B^*$ is the hyperplane defined by $x$.
    \end{itemize}
\end{lem}
\begin{proof}
If $p$ is not $\kk^2$-concise, then both claims trivially hold (except if rank of $p$ is one, then only the first claim holds).
Thus without loss of generality, we may suppose $p$ is concise by replacing $A$ and $B$ with smaller spaces if necessary.
If $\dim A\ge \dim B$, 
  then the projectivisation of the image 
  $\PP(p(A^*))\subset \PP(B\otimes \kk^2)$ intersects the Segre variety $\PP(B) \times \PP^1$ 
  by the dimension count \cite[Thm~I.7.2]{hartshorne} 
  (note that here we use that the base field $\kk$ is algebraically closed).
Otherwise, $\dim A < \dim B$ and the intersection 
  \[
    \PP(p(B^*)) \cap (\PP(A) \times \PP^1)\subset
        \PP(A\otimes \kk^2)
  \]
  has positive dimension by the same dimension count.
  In particular, any hyperplane $\PP(p(x^{\perp})) \subset \PP(p(B^*))$ also intersects the Segre variety.
\end{proof}

The next proposition reproves 
  (under the additional assumption that $\kk$ is algebraically closed) 
  and slightly strengthens the theorem of JaJa-Takche
  \cite{jaja_takche_Strassen_conjecture},
  which can be thought of as a theorem about $(0,2)$-hook shaped spaces.

\begin{prop}\label{prop_1_2_hook_shaped}
   Suppose $\kk$ is algebraically closed, $W''\subset B''\otimes C''$ is $(1,2)$-hook shaped and $W'\subset B'\otimes C'$ 
      is an arbitrary subspace.
   Then the additivity of the rank holds for $W'\oplus W''$.
\end{prop}

\begin{proof}
   We will use Lemmas~\ref{lem_proving_additivity_via_substitution},
   \ref{lem_constructing_next_tensor_in_the_sequence}
   and \ref{lem_first_step_for_hooks} again.
   That is, we are looking for a sequence 
      $p_0, \dotsc, p_{r''}\in A\otimes B \otimes C$ 
     with the properties 
     \ref{item_proof_additivity_hook_p0_eq_p}--\ref{item_proof_additivity_hook_p_i_drops_rank_enough},
   and the initial elements $\fromto{p_0}{p_k}$ are constructed in such a way that
   $p_k \in  A' \otimes B' \otimes C' 
           \oplus A''\otimes 
           \big(\linspan{x} \otimes C' \oplus
           B'' \otimes \kk^2 \big)$. 
   Here $x\in B''$ is such that 
     $W'' \subset \linspan{x} \otimes C'' +B''\otimes \kk^2$.

   We have already ``cleaned'' the part of the hook of size $1$, 
      and now we work with the remaining space of $\bfb''\times 2$ matrices.
   Unfortunately, cleaning $p''_i$ 
      produces rubbish in the other parts of the tensor, 
      and we have to  control the rubbish so that it does not affect $p'_i$, see \ref{item_proof_additivity_hook_p_prime_preserved}.
    Note that what is left to do is not just the plain case of Strassen's additivity in the case of $\bfc''=2$
    proven in \cite{jaja_takche_Strassen_conjecture} since $p_k$ may have already nontrivial entries 
    in another block, the one corresponding to $A''\otimes B''\otimes C'$ 
      (the small tensor $q_k$ in the statement of Lemma~\ref{lem_first_step_for_hooks}).

   We set $\Sigma =A'\otimes B' \otimes C'    \oplus     A\otimes \left(B\otimes \kk^2     \oplus    \linspan{x} \otimes C'\right)$.
  To construct $p_{j+1}$ we use Lemma~\ref{lem_dim_2_have_rank_one_matrix} 
     (in particular, here we exploit the algebraic closedness of $\kk$). 
  Thus either there exists $\alpha \in (A'')^*$ such that $R(p_j''(\alpha)) =1$, 
    or there exists $\beta \in x^{\perp} \subset (B'')^*$ such that $R(p_j''(\beta)) =1$.
  In both cases we apply 
  Lemma~\ref{lem_constructing_next_tensor_in_the_sequence}
  with the roles of $A$ and $C$ swapped or the roles of $B$ and $C$ swapped.
  The conditions in the lemma are straightforward to verify.

We stop after constructing $p_{r''}$ and thus the desired sequence exists and proves the claim.
\end{proof}

\section{Rank one matrices and additivity of the tensor rank}\label{sec_rank_one_matrices_and_additive_rank}

As hinted by the proof of Proposition~\ref{prop_1_2_hook_shaped},
   as long as we have a rank one matrix in the linear space $W'$ or $W''$, we have a good starting point for an attempt 
   to prove the additivity of the rank.
Throughout this section we will make a formal statement out of this observation and prove that if there is a rank one matrix in the linear spaces,
   then either the additivity holds or there exists a ``smaller'' example of failure of the additivity. 
In Section~\ref{sect_proofs_of_main_results_on_rank} we exploit several versions of this claim 
   in order to prove Theorem~\ref{thm_additivity_rank_intro}.

Throughout this section we follow Notations~\ref{notation_V_Seg} (denoting the rank one elements in a vector space by 
                                                                             the subscript $\cdot_{Seg}$),
    \ref{notation} (introducing the vector spaces $\fromto{A}{C''}$ and their dimensions $\fromto{\bfa}{\bfc''}$),
    \ref{notation_2} (defining a direct sum tensor $p=p'\oplus p''$ and the corresponding vector spaces $W, W', W''$),
       and also~\ref{notation_projection} (which explains the conventions for projections $\fromto{\pi_{A'}}{\pi_{C''}}$ 
       and vector spaces $\fromto{E'}{F''}$, which measure how much the fixed decomposition $V$ of $W$ sticks out 
       from the direct sum $B'\otimes C' \oplus B''\otimes C''$).

\subsection{Combinatorial splitting of the decomposition}

We carefully analyse 
   the structure of the rank one matrices in $V$.
We will distinguish seven types of such matrices.

\begin{lem}\label{lemma_reduction}
Every element of $V_{\Seg}\subset \PP(B\ts C)$  lies in the projectivisation of one of the following subspaces
  of $B\ts C$:
\begin{enumerate}
\item \label{item_prime_bis}
$B'\otimes C'$, $B''\otimes C''$,  \hfill (\Prime, \Bis)
\item  \label{item_vertical_and_horizontal}
$E'\otimes (C'\oplus F'')$,
$E''\otimes (F'\oplus C'')$,\hfill (\HL, \HR)\\
$(B'\oplus E'')\otimes F'$,
$(E'\oplus B'')\otimes F''$,\hfill (\VL, \VR) \item \label{item_mixed}
$(E'\oplus E'')\otimes (F'\oplus F'')$. \hfill (\Mix)
\end{enumerate}
\end{lem}
The spaces in \ref{item_prime_bis} are purely contained in the original direct summands, hence, in some sense, they are the easiest to deal with (we will show how to ``get rid'' of them and construct a smaller example justifying a potential lack of additivity).%
\footnote{The word \Bis{} comes from the Polish way of pronouncing the ${}''$ symbol.}
The spaces in \ref{item_vertical_and_horizontal} 
  stick out of the original summand, but only in one direction, either horizontal (\HL, \HR), or vertical (\VL, \VR)\footnote{Here the letters ``H, V, L, R'' stand for 
  ``horizontal, vertical, left, right'' respectively.}.  
The space in \ref{item_mixed} is mixed and it sticks out in all directions. It is the most difficult to deal with and we expect that the typical counterexamples to the additivity of the rank will have mostly (or only) such mixed matrices in their minimal decomposition. 
The mutual configuration and layout of those spaces in the case $\bfb',\bfb'',\bfc',\bfc''=3$, $\bfe',\bfe'',\bff',\bff''=1$ is illustrated in Figure~\ref{fig_layout_of_Prim__Mixed}.

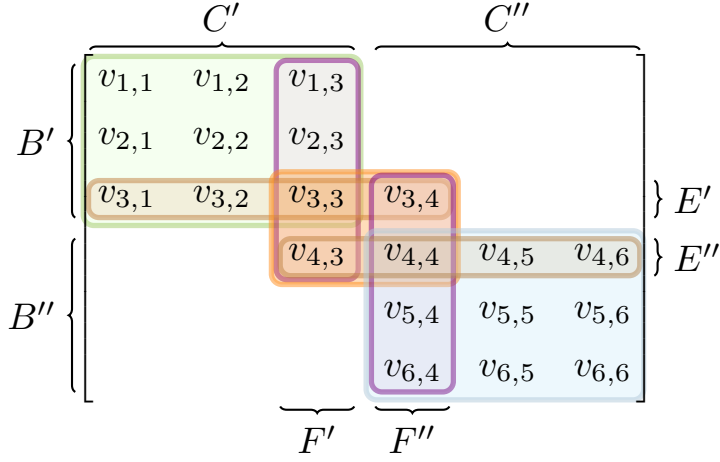
\begin{figure}
\centering
\usetikzlibrary{matrix,decorations.pathreplacing,calc,fit,backgrounds}

\pgfmathsetmacro{\myscale}{1.4}
\pgfkeys{tikz/mymatrixenv/.style={decoration={brace},every left delimiter/.style={xshift=8pt},every right delimiter/.style={xshift=-8pt}}}
\pgfkeys{tikz/mymatrix/.style={matrix of math nodes,nodes in empty cells,
left delimiter={[},right delimiter={]},inner sep=1pt,outer sep=1.5pt,
column sep=8pt,row sep=8pt,nodes={minimum width=20pt,minimum height=10pt,
anchor=center,inner sep=0pt,outer sep=0pt,scale=\myscale,transform shape}}}
\pgfkeys{tikz/mymatrixbrace/.style={decorate,thick}}

\newcommand*\mymatrixbraceright[4][m]{
    \draw[mymatrixbrace] (#1.west|-#1-#3-1.south west) -- node[left=2pt] {#4} (#1.west|-#1-#2-1.north west);
}
\newcommand*\mymatrixbraceleft[4][m]{
    \draw[mymatrixbrace] (#1.east|-#1-#2-1.north east) -- node[right=2pt] {#4} (#1.east|-#1-#2-1.south east);
}
\newcommand*\mymatrixbracetop[4][m]{
    \draw[mymatrixbrace] (#1.north-|#1-1-#2.north west) -- node[above=2pt] {#4} (#1.north-|#1-1-#3.north east);
}
\newcommand*\mymatrixbracebottom[4][m]{
    \draw[mymatrixbrace] (#1.south-|#1-1-#2.north east) -- node[below=2pt] {#4} (#1.south-|#1-1-#3.north west);
}

\tikzset{greenish/.style={
    fill=green!50!lime!60,draw opacity=0.5,
    draw=black!70!green!50!lime!60,fill opacity=0.1,
  },
  cyanish/.style={
    fill=cyan!90!blue!60, draw opacity=0.5,
    draw=black!70!blue!70!cyan!30,fill opacity=0.1,
  },
  orangeish/.style={
    fill=orange!90, draw opacity=0.5,
    draw=orange!90, fill opacity=0.3,
  },
  brownish/.style={
    fill=brown!70!orange!40, draw opacity=0.5,
    draw=brown, fill opacity=0.3,
  },
  purpleish/.style={
    fill=violet!90!pink!20, draw opacity=0.5,
    draw=violet, fill opacity=0.3,    
  }}

\[
%    \mathbf{X} = 
    \begin{tikzpicture}[baseline={-0.5ex},mymatrixenv]
        \matrix [mymatrix,inner sep=4pt] (m)  
        {
    v_{1,1}  &  v_{1,2} & v_{1,3}   &  \textcolor{white}{v_{1,4}}   &  &  \textcolor{white}{v_{1,6}} \\
    v_{2,1}  & v_{2,2} & v_{2,3} & & &   \\
    v_{3,1}  & v_{3,2} &  v_{3,3} & v_{3,4} & &   \\
    \textcolor{white}{v_{4,1}}    & &  v_{4,3} &  v_{4,4}  & v_{4,5}  & v_{4,6}   \\
    & & & v_{5,4} & v_{5,5} & v_{5,6}   \\
    \textcolor{white}{v_{6,1}}  & &  & v_{6,4}  & v_{6,5} & v_{6,6}   \\    
    };

    \begin{scope}[on background layer,rounded corners,every node/.append style={line width=2pt}]
     \node [fit=(m-1-1) (m-3-3),greenish,inner xsep=2.5pt,inner ysep=3pt]{};
     \node [fit=(m-1-3) (m-4-3),purpleish,inner xsep=1.5pt,inner ysep=1.5pt]{};
     \node [fit=(m-3-1) (m-3-4),brownish,inner xsep=0.5pt,inner ysep=0pt]{};
     \node [fit=(m-3-3) (m-4-4),orangeish]{};
     \node [fit=(m-3-4) (m-6-4),purpleish,inner xsep=1.5pt,inner ysep=1.0pt,yshift=1pt]{};
     \node [fit=(m-4-3) (m-4-6),brownish,inner xsep=0pt,inner ysep=0pt]{};
     \node [fit=(m-4-4) (m-6-6),cyanish,inner xsep=2.5pt,inner ysep=3pt,xshift=-1pt]{};
    \end{scope}

        % Braces     
    \begin{scope}[every node/.append style={scale=\myscale,transform
    shape},very thick]
        \mymatrixbraceright{1}{3}{$B'$}
        \mymatrixbraceright{4}{6}{$B''$}
        \mymatrixbracetop{1}{3}{$C'$}
        \mymatrixbracetop{4}{6}{$C''$}
        \mymatrixbracebottom{3}{3}{$F'$}
        \mymatrixbracebottom{4}{4}{$F''$}
        \mymatrixbraceleft{3}{3}{$E'$}
        \mymatrixbraceleft{4}{4}{$E''$}
    \end{scope} 
\end{tikzpicture}
\]
\caption{We use Notation~\ref{notation_projection}. 
In the case $\bfb',\bfb'',\bfc',\bfc''=3$, $\bfe',\bfe'',\bff',\bff''=1$
choose a basis of $E'$ and a completion to a basis of $B'$ and, 
similarly, bases for  $(E'', B''), (F', C'),(F'',C'')$. We can represent the elements of $V_{\Seg} \subset B \otimes C$ as matrices in one of the following subspaces:
\Prime{} (corresponding to the top-left green rectangle), 
\Bis{} (bottom-right blue rectangle), 
\VL{} (purple with entries $v_{1,3},\dots,v_{4,3}$), 
\VR{} (purple with entries $v_{3,4},\dots,v_{6,4}$), 
\HL{} (brown with entries $v_{3,1},\dots,v_{3,4}$),
\HR{} (brown with entries $v_{4,3},\dots,v_{4,6}$),
and \Mix{} (middle orange square with entries $\fromto{v_{3,3}}{v_{4,4}}$).
} \label{fig_layout_of_Prim__Mixed}
\end{figure}

\begin{proof}[Proof of Lemma~\ref{lemma_reduction}]
Let $b\ts c\in V_{\Seg}$ be a matrix of rank one. Write $b=b'+b''$ and $c=c'+c''$, where $b'\in B', b''\in B'', c'\in C'$ and $c''\in C''$.  We consider the image of $b\ts c$ via the four natural projections introduced in 
  Notation~\ref{notation_projection}:
\begin{subequations}
  \begin{alignat}{3}
     \pi_{B' }(b\ts c)&=b''\ts c  & \ \in \ && B''&\ts(F' \oplus C''),
     \label{equ_pi_B_prime}\\
     \pi_{B''}(b\ts c)&=b' \ts c  & \ \in \ && B' &\ts(C' \oplus F''),
     \label{equ_pi_B_bis}\\
     \pi_{C' }(b\ts c)&=b  \ts c''& \ \in \ && (E'\oplus B'')&\ts C'', \text{ and}
     \label{equ_pi_C_prime}\\
     \pi_{C''}(b\ts c)&=b  \ts c' & \ \in \ && (B'\oplus E'')&\ts C'.
     \label{equ_pi_C_bis}
  \end{alignat}
\end{subequations}
Notice that $b'$ and $b''$ cannot be simultaneously zero, since $b\ne 0$.
Analogously, $(c',c'')\neq (0,0)$.

Equations~\eqref{equ_pi_B_prime}--\eqref{equ_pi_C_bis} prove that the non-vanishing of one of $b', b'', c', c''$ induces a restriction on another one. For instance, if $b'\ne 0$, then by \eqref{equ_pi_B_bis} we must have $c''\in F''$.
Or, if $b''\ne 0$, then \eqref{equ_pi_B_prime} forces $c' \in F'$,
  and so on.
Altogether we obtain the following cases:
\begin{itemize}
\item[(1)] If $b',b'',c',c''\neq0$, then  $b\ts c\in (E'\oplus E'')\ts (F'\oplus F'')$ (case \Mix).
\item[(2)] if $b',b''\neq 0$ and $c'=0$, then $b\ts c=b\ts c''\in (E'\oplus B'')\ts  F''$ (case \VR).
\item[(3)] if $b',b''\neq0$ and $c''=0$, then $b\ts c=b\ts c'\in (B'\oplus E'')\ts  F'$ (case \VL).
\item[(4)] If $b'=0$, then either $c'=0$ and therefore $b\ts c=b''\ts c''\in B''\ts C''$ (case \Bis), or $c'\neq0$ and $b\ts c=b''\ts c\in E''\ts (F'\oplus C'')$ (case \HR).
\item[(5)] If $b''=0$, then either $c''=0$ and thus $b\ts c=b'\ts c'\in B'\ts C'$ (case \Prime), or $c''\neq0$ and $b\ts c=b''\ts c\in E'\ts (C'\oplus F'')$ (case \HL).
\end{itemize}
This concludes the proof.
\end{proof}

As in Lemma~\ref{lemma_reduction} every element of $V_{Seg}\subset \PP(B\otimes C)$ 
   lies in one of seven subspaces of $B\otimes C$. 
These subspaces may have nonempty intersection.
We will now explain our convention with respect to choosing a basis of $V$ consisting of elements of $V_{\Seg}$.

Here and throughout the article by $\sqcup$ we denote the disjoint union.

\begin{notation}\label{notation_Prime_etc}
  We choose a basis $\ccB$ of $V$ in such a way that:
  \begin{itemize}
   \item $\ccB$ consist of rank one matrices only,
   \item $\ccB = \Prime \sqcup \Bis \sqcup \HL \sqcup \HR \sqcup \VL \sqcup \VR \sqcup \Mix$, 
           where each of \Prime, \Bis, \HL, \HR, \VL, \VR, and \Mix{} 
           is a finite set of rank one matrices of the respective type as in Lemma~\ref{lemma_reduction} 
           (for instance, $\Prime \subset B'\otimes C'$, $\HL\subset E'\otimes (C'\oplus F'')$, etc.).
   \item $\ccB$ has as many elements of $\Prime$ and $\Bis$ as possible, subject to the first two conditions,
   \item $\ccB$ has as many elements of $\HL$, $\HR$, $\VL$ and $\VR$ as possible, subject to all of the above conditions.
  \end{itemize}
  Let $\prim$ be the number of elements of $\Prime$ (equivalently, $\prim = \dim\linspan{\Prime}$)
     and analogously define $\bis$, $\hl$, $\hr$, $\vl$, $\vr$, and $\mix$.
  The choice of $\ccB$ need not  be unique, but we fix one for the rest of the article.
  Instead, the numbers $\prim$, $\bis$, and $\mix$ are uniquely determined by $V$
  (there may be some non-uniqueness in dividing between  $\hl$, $\hr$, $\vl$, $\vr$).
\end{notation}

Thus to each decomposition we associated a sequence of seven non-negative integers $(\fromto{\prim}{\mix})$.
We now study the inequalities between these integers and exploit them to get theorems about the additivity of the rank.
  
\begin{prop}\label{prop_projection_inequalities}
 In Notations~\ref{notation_projection} and \ref{notation_Prime_etc} the following inequalities hold:
 \begin{enumerate}
  \item \label{item_projection_inequality_short_prim}
        $\prim + \hl + \vl 
  + \min\big(\mix,\bfe' \bff'  \big) \geq R(W')$,
  \item \label{item_projection_inequality_short_bis}
        $\bis + \hr  + \vr 
  + \min\big(\mix,\bfe'' \bff'' \big) \geq R(W'')$,
  \item \label{item_projection_inequality_long_prim_H}
        $\prim  + \hl + \vl 
  + \min\big(\hr + \mix , \bff'(\bfe'+\bfe'')\big) \geq R(W') + \bfe''$,
  \item \label{item_projection_inequality_long_prim_V}
        $\prim + \hl +  \vl 
  + \min\big(\vr +\mix,\bfe'(\bff'+\bff'')\big) \geq R(W') + \bff''$,
  \item \label{item_projection_inequality_long_bis_H}
        $\bis +   \hr + \vr 
  + \min\big(\hl +\mix,\bff''(\bfe'+\bfe'')\big) \geq R(W'') + \bfe'$,
  \item \label{item_projection_inequality_long_bis_V}
        $\bis + \hr + \vr   
  + \min\big(\vl+ \mix,\bfe''(\bff'+\bff'')\big) \geq R(W'') + \bff'$.
 \end{enumerate}
\end{prop}
\begin{proof}
   To prove Inequality~\ref{item_projection_inequality_short_prim}   
      we consider the composition of projections $\pi_{B''}\pi_{C''}$.
   The linear space $\pi_{B''}\pi_{C''}(V)$ is spanned by rank one matrices 
      $\pi_{B''}\pi_{C''}(\ccB)$ (where $\ccB = \Prime \sqcup \dotsb \sqcup \Mix$ as in Notation~\ref{notation_Prime_etc}), 
      and it contains $W'$.
   Thus $\dim (\pi_{B''}\pi_{C''}(V)) \ge R(W')$. 
   But the only elements of the basis $\ccB$
      that survive both projections (that is, they are not mapped to zero under the composition) 
      are \Prime, \HL, \VL, and \Mix.
   Thus 
   \[
        \prim + \hl + \vl + \mix \geq \dim (\pi_{B''}\pi_{C''}(V)) \ge R(W').
   \]
   On the other hand, $\pi_{B''}\pi_{C''}(\Mix) \subset E'\otimes F'$, 
     thus among $\pi_{B''}\pi_{C''}(\Mix)$ we can choose at most $\bfe'\bff'$ linearly independent matrices.
   Thus 
   \[
        \prim + \hl + \vl + \bfe'\bff' \geq \dim (\pi_{B''}\pi_{C''}(V)) \ge R(W').
   \]
  The two inequalities prove \ref{item_projection_inequality_short_prim}.

  To show Inequality~\ref{item_projection_inequality_long_prim_H}
    we may assume that $W'$ is concise as in the proof of Lemma~\ref{lemma_bound_r'_e'_R_w'}.
  Moreover, as in that same proof (more precisely, Inequality~\eqref{equ_bound_on_pi_C_bis_of_V}) we show that
  $
      \dim\pi_{C''}(V) \ge  R(W') +\bfe''.
  $
  But $\pi_{C''}$ sends all matrices from $\Bis$ and $\VR$ to zero, thus
  \[
     \prim  + \hl +  \vl + \hr + \mix \geq \dim\pi_{C''}(V) \ge R(W') + \bfe''.
  \]
  As in the proof of Part~\ref{item_projection_inequality_short_prim}, we can also replace $\hr + \mix$ by $\bff'(\bfe'+\bfe'')$, 
    since $\pi_{C''}(\HR\cup \Mix) \subset (E' \oplus E'')\otimes F'$, 
    concluding the proof of \ref{item_projection_inequality_long_prim_H}.
    
 The proofs of the remaining four inequalities are identical to one of the above,
    after swapping the roles of $B$ and $C$ or ${}'$ and ${}''$ (or swapping both pairs).
\end{proof}

\begin{prop}\label{prop_SAC_if_E'=0}
 With Notation~\ref{notation_projection}, 
   if one among $E',E'',F',F''$ is zero, then $R(W)=R(W')+R(W'')$.
 %(Analogously, if one of $E''$, $F'$, $F''$ is zero,   the same conclusion holds.)  
\end{prop}
\begin{proof}
  Let us assume without loss of generality that $E'=\{0\}$. 
  Using the definitions of sets $\Prime$, $\Bis$, $\VR$,\dots as in  Notation~\ref{notation_Prime_etc}
    we see that $\HL = \VR= \Mix= \emptyset$, due to the order of choosing the elements of the basis $\ccB$:
  For instance, a potential candidate to became a member of \HL, would be first elected to $\Prime$, 
      and similarly $\VR$ is consumed by $\Bis$ and $\Mix$ by $\HR$.
  Thus:
 \begin{equation*}%\label{equation_cor_E'zero}
   R(W) =   \dim(V_{Seg})
        =   \prim + \bis + \hr + \vl.
 \end{equation*}
 Proposition~\ref{prop_projection_inequalities}\ref{item_projection_inequality_short_prim} and \ref{item_projection_inequality_short_bis} implies
 \[
    R(W') +R(W'') \leq \prim + \vl + \bis + \hr  = R(W), \\    
 \]
 while $R(W') +R(W'') \ge R(W)$ always holds. This shows the desired additivity. 
\end{proof}

\renewcommand{\theenumi}{(\alph{enumi})}

\begin{cor}\label{cor_inequalities_HL__Mix}
  Assume that the additivity fails for $W'$ and $W''$, that is, $d=R(W') + R(W'')-R(W'\oplus W'') >0$.
  Then the following inequalities hold:
 \begin{enumerate}
    \item \label{item_mix_at_least_1}
          $\mix \geq d \ge 1$, 
    \item \label{item_mix_and_horizontal_at_least_1_plus_e}
        $\hl + \hr + \mix \geq \bfe' + \bfe'' + d \ge 3$,
    \item \label{item_mix_and_vertical_at_least_1_plus_f}
        $\vl + \vr + \mix \geq \bff' + \bff'' + d \ge 3$.
 \end{enumerate}
\end{cor}
 
\renewcommand{\theenumi}{(\roman{enumi})}

\begin{proof} 
   To prove \ref{item_mix_at_least_1} consider the inequalities \ref{item_projection_inequality_short_prim} and \ref{item_projection_inequality_short_bis} from Proposition~\ref{prop_projection_inequalities} and their sum:
   \begin{align}
       \prim + \hl + \vl + \mix &\ge R(W' ),\nonumber \\ 
       \bis  + \hr + \vr + \mix &\ge R(W''),\nonumber\\
       \prim + \bis +\hl + \hr + \vl +\vr + 2 \mix 
                    &\ge R(W') + R(W'') \label{equ_mix_at_least_1_plus_Rank}.
   \end{align}
   The lefthand side of \eqref{equ_mix_at_least_1_plus_Rank} is equal to $R(W)+ \mix$,
   while its righthand side is $R(W)+d$. 
   Thus the desired claim. % follows from \eqref{equ_mix_at_least_1_plus_Rank} by subtracting $R(W)$ from both sides.
 
   Similarly, using inequalities \ref{item_projection_inequality_long_prim_H} 
       and \ref{item_projection_inequality_long_bis_H} 
       of the same propostion %(and subtracting $R(W)$ from both sides)
       we obtain \ref{item_mix_and_horizontal_at_least_1_plus_e},
   while \ref{item_projection_inequality_long_prim_V}
       and \ref{item_projection_inequality_long_bis_V} imply~\ref{item_mix_and_vertical_at_least_1_plus_f}.
   Note that $\bfe' + \bfe'' + d \ge 3$ and  $\bff' + \bff'' + d \ge 3$
       by Proposition~\ref{prop_SAC_if_E'=0}.
\end{proof}

\subsection{Replete pairs}

As we hunger after inequalities involving integers $\fromto{\prim}{\mix}$ we distinguish a class of pairs $W', W''$ 
   with particularly nice properties.
\begin{defin}
   Consider a pair of linear spaces $W'\subset B'\otimes C'$ and $W''\subset B'' \otimes C''$
     with a fixed minimal decomposition $V= \linspan{V_{\Seg}} \subset B\otimes C$ and $\fromto{\Prime}{\Mix}$
     as in Notation~\ref{notation_Prime_etc}.
   We say $(W', W'')$ is \emph{replete}, if $\Prime \subset W'$ and $\Bis \subset W''$.
\end{defin}

\begin{rem}
   Striclty speaking, the notion of \emph{replete pair} depends also on the minimal decomposition $V$. 
   But as always we consider a pair $W'$ and $W''$ with a fixed decomposition
     $V=\linspan{V_{\Seg}} \supset W'\oplus W''$, so we refrain from mentioning $V$ in the notation.
\end{rem}

The first important observation is that as long as we look for pairs that fail to satisfy the additivity, we are free to replenish any pair. 
More precisely, for any fixed $W'$, $W''$ (and $V$) define the \emph{repletion} of $(W', W'')$  as the pair $(\repletion{W'},\repletion{W''})$:
\begin{equation}
\begin{aligned}
   \repletion{W'}:&=W'+\linspan{\Prime}, & 
   \repletion{W''}:&=W''+\linspan{\Bis}, &  
   \repletion{W}:&=\repletion{W'}\oplus \repletion{W''}.
\end{aligned}
\end{equation}

\begin{prop}\label{prop_does_not_hurt_to_replenish}
   For any $(W', W'')$, with Notation~\ref{notation_Prime_etc},
     we have:
   \begin{align*}
        R(W' ) \le R(\repletion{W' }) & \le R(W' ) + (\dim \repletion{W' } - \dim W' ),\\
        R(W'') \le R(\repletion{W''}) & \le R(W'') + (\dim \repletion{W''} - \dim W''),\\
        R(\repletion{W}) & = R(W).
   \end{align*}
   In particular, if the additivity of the rank fails for $(W',W'')$, then it also fails for 
      $(\repletion{W'},\repletion{W''})$. 
   Moreover,
   \begin{enumerate}
    \item \label{item_repletion_has_the_same_decomposition} 
          $V$ is a minimal decomposition of $\repletion{W}$; 
             in particular, the same distinguished basis 
             $\Prime\sqcup \Bis \sqcup\dotsb\sqcup \Mix$ works for both $W$ and  $\repletion{W}$.
    \item \label{item_repletion_is_replete}
          $(\repletion{W'}, \repletion{W''})$ is a replete pair.
    \item \label{item_gap_is_preserved_under_repletion}
          The gaps $\gap{\repletion{W'}}$, $\gap{\repletion{W''}}$, and  $\gap{\repletion{W}}$, 
               are at most (respectively) 
               $\gap{W'}$, $\gap{W''}$, and  $\gap{W}$.
   \end{enumerate}
\end{prop}

\begin{proof}
    Since $W' \subset \repletion{W'}$, the inequality $R(W')\le R(\repletion{W'})$ is clear.
    Moreover $\repletion{W'}$ is spanned by $W'$ and $(\dim \repletion{W' } - \dim W' )$ 
       additional matrices, that can be chosen out of $\Prime$ --- in particular, these additional matrices are all of rank $1$
       and  $R(\repletion{W'}) \le  R(W') + (\dim \repletion{W' } - \dim W' )$.
    The inequalities about ${}''$ and $R(W)\le R(\repletion{W})$ follow similarly.
    
    Further $\repletion{W} \subset V$, thus $V$ is a decomposition of $\repletion{W}$.
    Therefore also $R(\repletion{W}) \le \dim V = R(W)$, showing $R(\repletion{W})  = R(W)$ 
       and \ref{item_repletion_has_the_same_decomposition}.
    Item~\ref{item_repletion_is_replete} follows from \ref{item_repletion_has_the_same_decomposition},
       while \ref{item_gap_is_preserved_under_repletion} is a rephrasement of the initial inequalities.
\end{proof}

Moreover, if one of the inequalities of Lemma~\ref{lemma_bound_r'_e'_R_w'} is an equality, 
    then the respective $W'$ or $W''$ is not affected by the repletion.
    
\begin{lem}\label{lem_minimal_pair_is_replete}
    If, say, $R(W') + \bfe'' = R(W) - \dim W''$, then $W'' = \repletion{W''}$, 
       and analogous statements hold for the other equalities coming from replacing $\le$ by $=$ in 
       Lemma~\ref{lemma_bound_r'_e'_R_w'}.
\end{lem}

\begin{proof}
    By Lemma~\ref{lemma_bound_r'_e'_R_w'} applied to $\repletion{W} = \repletion{W'}\oplus \repletion{W''}$ 
      and by Proposition~\ref{prop_does_not_hurt_to_replenish}
    we have:
    \begin{align*}
        R(\repletion{W}) - \bfe'' & \stackrel{\text{\ref{lemma_bound_r'_e'_R_w'}}}{\ge} R(\repletion{W'}) + \dim (\repletion{W''}) \\
        & \stackrel{\text{\ref{prop_does_not_hurt_to_replenish}}}{\ge} R(W') + \dim W'' \\ 
        &\stackrel{\text{assumptions of \ref{lem_minimal_pair_is_replete}}}{=}\ \ R(W) - \bfe'' 
        \stackrel{\text{\ref{prop_does_not_hurt_to_replenish}}}{=} R(\repletion{W}) - \bfe''.
    \end{align*}
    Therefore all inequalities are in fact equalities. In particular,  $\dim (\repletion{W''}) = \dim W''$.
    The claim of the lemma follows from $W'' \subset \repletion{W''}$.
\end{proof}

% 
% \begin{cor}
%     Suppose $R(W'') \le \dim W'' +2 $. Then either the additivity holds $R(W'\oplus W'') = R(W') + R(W'')$ 
%        or:
%     \begin{itemize}
%      \item $R(W'') = \dim W'' +2 $, and 
%      \item $R(W'\oplus W'') = R(W') + R(W'')-1$, and 
%      \item $\bfe'' = \bff''=1$, and 
%      \item $\repletion{W''} = W''$.
%     \end{itemize}
% \end{cor}
% 
% 
% \begin{proof}
%    Suppose that the additivity does not hold. 
%    Then by Lemma~\ref{lem_rank_at_least_2_more_than_dimension} we must have $R(W'') = \dim W'' +2 $ 
%       and $\bfe''> 0$, $\bff''> 0$, while by Corollary~\ref{cor_bounds_on_es_and_fs} 
%       we must have $\bfe''< 2$ and $\bff'' < 2$.
%    Thus $\bfe'' =\bff'' = 1$.
%   
%    By Lemma~\ref{lemma_bound_r'_e'_R_w'} the inequality $R(W' \oplus W'') \ge R(W')+ 1 + \dim W''$ holds.
%    The right hand side is equal to $R(W') + R(W'')-1$ by the above discussion 
%       (the $\le$ inequality follows from the failure of additivity).
% 
%    The final claim $\repletion{W''} = W''$ follows from Lemma~\ref{lem_minimal_pair_is_replete}.
% \end{proof}

\subsection{Digestion}

%After a good meal the digestive process begins.
For replete pairs it makes sense to consider the complement of $\linspan{\Prime}$ in $W'$, 
   and of $\linspan{\Bis}$ in $W''$.

\begin{defin} With Notation~\ref{notation_Prime_etc}, 
    suppose $S'$ and $S''$ denote the following linear spaces:
    \begin{align*}
      S':&=\linspan{\Bis \sqcup \HL \sqcup \HR \sqcup \VL \sqcup \VR \sqcup \Mix} 
         \cap W' && \text{(we omit $\Prime$ in the union) and }\\
      S'':&=\linspan{\Prime\sqcup \HL\sqcup \HR\sqcup \VL\sqcup \VR\sqcup \Mix} \cap W'' && \text{(we omit $\Bis$ in the union).}
    \end{align*}
  We call the pair $(S', S'')$ the \emph{digested version} of $(W',W'')$.  
\end{defin}

\begin{lem}\label{lemma_dim_S'}
 If $(W', W'')$ is replete,
   then $W'  = \linspan{\Prime} \oplus S'$ and $W''  = \linspan{\Bis} \oplus S''$.
\end{lem}
\begin{proof}
Both $\linspan{\Prime}$ and $S'$ are contained in $W'$. 
The intersection $\linspan{\Prime} \cap S'$ is zero, since 
   the seven sets $\Prime, \Bis, \HR, \HL, \VL, \VR, \Mix$ are disjoint and together they are linearly independent.
Furthermore,  
\[
   \codim (S' \subset W') \le \codim (\linspan{\Bis \sqcup \HR \sqcup \HL \sqcup \VL \sqcup \VR \sqcup \Mix} \subset V) =  \prim.
\]
Thus $\dim S' + \prim \ge \dim W'$, which concludes the proof of the first claim. The second claim is analogous.
\end{proof}

These complements $(S', S'')$ might replace the original replete pair $(W',W'')$:
  as we will show, if the additivity of the rank fails for $(W',W'')$, it also fails for $(S', S'')$.
Moreover, $(S', S'')$ is still replete, but it does not involve any $\Prime$ or $\Bis$.

\begin{lem}\label{lemma_can_digest}
Suppose $(W', W'')$ is replete, define $S'$ and $S''$ as above and set $S=S'\oplus S''$.
Then
\begin{enumerate}
 \item \label{item_can_digest___rank}
       $R(S)=R(W) -\prim-\bis = \hl+\hr+\vl+\vr+\mix$
         and the space $\linspan{\HL,\HR,\VL,\VR,\Mix}$ determines a minimal decomposition of $S$.
       In particular, $(S',S'')$ is replete and both spaces $S'$ and $S''$ contain no rank one matrices.
 \item \label{item_can_digest___additivity}
       If the additivity of the rank $R(S) =  R(S')+ R(S'')$ holds for $S$,
         then it also holds for $W$, that is, $R(W) =  R(W')+ R(W'')$.
\end{enumerate}
\end{lem}
\begin{proof}
     Since $W = S\oplus \linspan{\Prime, \Bis}$, we must have $R(W)\le R(S)+ \prim +\bis$.
     On the other hand, $S\subset \linspan{\HL,\HR,\VL,\VR,\Mix}$, hence $R(S)\le \hl+\hr+\vl+\vr+\mix$.
     These two claims show the equality for $R(S)$ in \ref{item_can_digest___rank} 
        and that $\linspan{\HL,\HR,\VL,\VR,\Mix}$ gives a minimal decomposition of $S$.
     Since there is no tensor of type $\Prime$ or $\Bis$ in this minimal decomposition, 
        it follows that the pair $(S',S'')$ is replete by definition.
     If, say, $S'$ contained a rank one matrix, then by our choice of basis in 
        Notation~\ref{notation_Prime_etc} it would be in the span of $\Prime$, a contradiction.

     Finally, if $R(S) =  R(S')+ R(S'')$, then:
     \begin{align*}
        R(W) &= R(S)+ \prim + \bis \\
             &= R(S')+ \prim + R(S'')+\bis \ge R(W') + R(W''),
     \end{align*}
     showing the statement  \ref{item_can_digest___additivity} for $W$.
\end{proof}

As a summary, in our search for examples of failure of the additivity of the rank,
   in the previous section we replaced a linear space $W=W'\oplus W''$ by its
repletion $\repletion{W} = \repletion{W'} \oplus \repletion{W''}$, that is possibly larger.
Here in turn, we replace  $\repletion{W}$ by a smaller linear space $S = S' \oplus S''$.
In fact, $\dim W'\ge \dim S'$ and $\dim W''\ge \dim S''$, and also $R(S)\le R(W)$ and $R(S')\le R(W')$ etc.
That is, changing $W$ into $S$ makes the corresponding tensors possibly ``smaller'',
   but not larger.
In addition, we gain more properties: $S$ is replete and has no $\Prime$'s or $\Bis$'s in its minimal decomposition.

\begin{cor}\label{cor_small_dimensions_of_E_and_F}
   Suppose that $W =  W'\oplus W''$ is as in Notation~\ref{notation_2} 
      and that $\bfe''$ and $\bff''$ are as in Notation~\ref{notation_projection}. 
   If either:
   \begin{enumerate}
    \item\label{item_cor_small_dims_of_E_and_F_arbitrary_field}
    $\kk$ is an arbitrary field, $\bfe''\le 1$ and $\bff''\le 1$, or
    \item\label{item_cor_small_dims_of_E_and_F_alg_closed_field}
    $\kk$ is algebraically closed, $\bfe''\le 1$ and $\bff''\le 2$,
   \end{enumerate}
 then the additivity of the rank $R(W) =  R(W')+ R(W'')$ holds.
\end{cor}
\begin{proof}
   By Proposition~\ref{prop_does_not_hurt_to_replenish} and Lemma~\ref{lemma_can_digest}, 
      we can assume $W$ is replete and equal to its digested version.
   But then (since $\Bis =\emptyset$) we must have $W''\subset E''\otimes C'' + B''\otimes F''$.
   In particular, $W''$ is, respectively, a $(1,1)$-hook shaped space or a $(1,2)$-hook shaped space.
   Then the claim follows from Proposition~\ref{prop_1_1_hook_shaped} or Proposition~\ref{prop_1_2_hook_shaped}. 
\end{proof}

\subsection{Additivity of the tensor rank for small tensors}\label{sect_proofs_of_main_results_on_rank}

We conclude our discussion of the additivity of the tensor rank with the following summarising results.

\begin{thm}\label{thm_additivity_rank_plus_2}
   Over an arbitrary base field $\kk$ 
      assume $p'\in A'\otimes B'\otimes C'$ is any tensor, while  
      $p''\in A''\otimes B''\otimes C''$ is concise and $R(p'')\le \bfa''+2$.
   Then the additivity of the rank holds:
   \[
     R(p'\oplus p'') = R(p')+R(p'').
   \]
  The analogous statements with the roles of $A$ replaced by $B$ or $C$, or the roles of ${}'$ and ${}''$ 
      swapped, hold as well.
\end{thm}

\begin{proof} Since $p''$ is concise, the corresponding vector subspace $W''=p''((A'')^*)$ 
      has dimension equal to $\bfa''$.
   By Corollary~\ref{cor_small_dimensions_of_E_and_F}\ref{item_cor_small_dims_of_E_and_F_arbitrary_field} 
      we may assume $\bfe''\ge  2$ or  $\bff''\ge 2$.
   Say, $\bfe'' \ge 2 \ge R(p'') - \dim W''$,
      then by Corollary~\ref{cor_bounds_on_es_and_fs} the additivity must hold.
\end{proof}

\begin{thm}\label{thm_additivity_rank_in_dimensions_a_3_3}
   Suppose the base field is $\kk =\CC$ or $\kk=\RR$ 
      (complex or real numbers) and
      assume $p'\in A'\otimes B'\otimes C'$ is any tensor, while  
      $p''\in A''\otimes \kk^3\otimes \kk^3$ for an arbitrary vector space $A''$.
   Then the additivity of the rank holds:
$R(p'\oplus p'') = R(p')+R(p'')$.
\end{thm}

\begin{proof} 
By the classical Ja'Ja'-Takche Theorem~\cite{jaja_takche_Strassen_conjecture}
       (in the algebraically closed case also shown in Proposition~\ref{prop_1_2_hook_shaped}), 
      we can assume $p''$ is concise in $A''\otimes \kk^3\otimes \kk^3$.
   But then by \cite[Thm~5 and Thm~6]{sumi_miyazaki_sakata_maximal_tensor_rank} 
      the rank of $p''$ is at most $\bfa'' +2$
      and the result follows from Theorem~\ref{thm_additivity_rank_plus_2}.
\end{proof}

Note that in the proof above we exploit the results 
   about maximal rank in $\kk^{\bfa''} \otimes \kk^3\otimes \kk^3$.
In \cite{sumi_miyazaki_sakata_maximal_tensor_rank} the authors assume that the base field is $\CC$ o r$\RR$.
We are not aware of any similar results over other fields, 
  with the unique exception of $\bfa''=3$, see the next proof for a discussion.

\begin{thm}\label{thm_additivity_rank_6}
   Suppose the base field $\kk$ is such that:
     \begin{itemize}
      \item the maximal rank of a tensor in $\kk^3\otimes \kk^3 \otimes \kk^3$ is at most $5$.
     \end{itemize}
     (For example $\kk$ is algebraically closed of characteristic $\ne 2$ or $\kk= \RR$).
   Furthermore assume $R(p'') \le 6$. 
   Then independently of $p'$, the additivity of the rank holds: $R(p'\oplus p'')= R(p')+R(p'')$.
\end{thm}
\begin{proof}
   Without loss of generality, we may assume $p''$ is concise in $A''\otimes B''\otimes C''$. 
   As in the previous proof, if any of the dimensions $\dim A''$, $\dim B''$, $\dim C''$ is at most $2$, 
      then the claim follows from \cite{jaja_takche_Strassen_conjecture}.
   On the other hand, if any of the dimensions  $\bfa''$, $\bfb''$, $\bfc''$ is at least $4$, 
      then the result follows from Theorem~\ref{thm_additivity_rank_plus_2}.
   The remaining case $\bfa''=\bfb''=\bfc''=3$ also follows from Theorem~\ref{thm_additivity_rank_plus_2} by our assumption on the field $\kk$.
   
   The assumption is satisfied for $\kk=\RR, \CC$ 
      see \cite[Thm~5.1]{bremner_hu_Kruskal_theorem} 
      or \cite[Thm~5]{sumi_miyazaki_sakata_maximal_tensor_rank}.
   In \cite[top of p.~402]{bremner_hu_Kruskal_theorem} the authors say
      that their proof is also valid for any algebraically closed field
      of characteristic not equal to $2$.
  They also provide the interesting history of this question
     and, furthermore, they show that the assumption about maximal rank in 
     $\kk^3\times \kk^3\times \kk^3$ fails for $\kk=\ZZ_2$.
\end{proof}

Assuming the base field is $\kk=\CC$, 
   one of the smallest cases not covered by the above theorems would be
   the case of $p', p''\in \CC^4\otimes \CC^4\otimes \CC^3$.
The generic rank (that is, the rank of a general tensor) in $\CC^4\otimes \CC^4\otimes \CC^3$
   is $6$, moreover \cite[p.~6]{atkinson_stephens_maximal_rank_of_tensors} 
   claims the maximal rank is $7$ (see also \cite[Prop.~2]{sumi_miyazaki_sakata_maximal_tensor_rank}).
\begin{example}
   Suppose $A'=A''=\CC^4$ and either $B'=B''= \CC^4$ and $C'=C''= \CC^3$
      or $B'=C''= \CC^4$ and $B'=C''= \CC^3$.
   Suppose both $p'\in A'\otimes B'\otimes C'$
      and $p''\in A''\otimes B''\otimes C''$ are tensors of rank $7$ 
      and that the additivity of the rank fails for $p= p'\oplus p''$.
   Let $W'$ and $W''$ be as in Notation~\ref{notation_2}, 
      and $E'$, $\bfe'$, etc.~be as in Notation~\ref{notation_projection}.
   Then:
   \begin{itemize}
    \item $R(p) =13$,
    \item $\bfe'=\bfe'' = \bff' =\bff''=2$,
    \item with $\Prime$, $\hl$, etc., as in Notation~\ref{notation_Prime_etc}, 
          we have $\Prime =\Bis = \emptyset$, and the following inequalities hold:
          \[
           \begin{array}{|lcccl|}
\hline           
           \multicolumn{5}{|c|}{\text{if } \bfb''= 4, \bfc''=3}\\
\hline
\hline
              2 & \le & \hl  & \le &3\\
\hline
              2 & \le & \hr  & \le &3\\
\hline
              3 & \le & \vl  & \le &4\\
\hline
              3 & \le & \vr  & \le &4\\
\hline
              1 & \le & \mix & \le &3\\
\hline
                \multicolumn{3}{|r}{\hl+\vl}  & \le &6\\
\hline
                \multicolumn{3}{|r}{\hr+\vr} & \le &6\\
\hline
           \end{array}
\quad \text{ or } \quad 
           \begin{array}{|lcccl|}
\hline           
           \multicolumn{5}{|c|}{\text{if } \bfb''= 3, \bfc''=4}\\
\hline
\hline
              2 & \le & \hl  & \le &3\\
\hline
              3 & \le & \hr  & \le &4\\
\hline
              3 & \le & \vl  & \le &4\\
\hline
              2 & \le & \vr  & \le &3\\
\hline
              1 & \le & \mix & \le &3\\
\hline
                \multicolumn{3}{|r}{\hl+\vl}  & \le &6\\
\hline
                \multicolumn{3}{|r}{\hr+\vr} & \le &6.\\
\hline
           \end{array}
           \]
   \end{itemize}
\end{example}
\begin{proof}[Sketch of proof]
    For brevity we only argue in the case  $\bfb''= 4, \bfc''=3$, 
       while the proof of $\bfb''= 3, \bfc''=4$ is very similar.
    Both tensors $p', p''\in \CC^4 \otimes \CC^4\otimes \CC^3$ must be concise, 
       as otherwise either Theorem~\ref{thm_additivity_rank_in_dimensions_a_3_3} 
       or JaJa-Takche Theorem imply the additivity of the rank.
    By Corollary~\ref{cor_bounds_on_es_and_fs} we must have $\bfe' \le 2$, 
       and similarly for $\bff'$, $\bfe''$, $\bff''$.
    If one of them is strictly less then $2$, 
       then Corollary~\ref{cor_small_dimensions_of_E_and_F}\ref{item_cor_small_dims_of_E_and_F_alg_closed_field} implies the additivity, a contradiction,
       thus  $\bfe'=\bfe'' = \bff' =\bff''=2$.

    By the failure of the additivity, we must have $R(W)\le 13$, but Lemma~\ref{lemma_bound_r'_e'_R_w'} 
       implies also $R(W)\ge 13$, showing that $R(p)=13$.
       
    If, say $\Prime \ne \emptyset$, then the digested version $(S', S'')$ of repletion of $(W', W'')$ 
       is also a counterexample to the additivity by 
       Lemma~\ref{lemma_can_digest}\ref{item_can_digest___additivity}.
    If $S=S'\oplus S''$ has lower rank than $W$, then either  $S$ is not concise, 
       contradicting Theorem~\ref{thm_additivity_rank_in_dimensions_a_3_3}
       or $S$ contradicts the above calculations of rank.
    Thus also $R(S)=13$ and by Lemma~\ref{lemma_can_digest}\ref{item_can_digest___rank}
       we must have $\prim=\bis =0$. In fact, $S=W$.

    Let $\widetilde{E'}\subset E'$ be the smallest linear subspace such that 
        $\pi_{C''}(\HL)\subset  \widetilde{E'}\otimes C'$. Set ${\bf{\widetilde{e}'}}=\dim \widetilde{E'}$.
%         and analogously define $\widetilde{F'}$ using $\pi_{B''}(\VL)\subset B'\otimes \widetilde{F'}$.
    Since $\Prime=\emptyset$, we must have 
    \[
        W'\subset \linspan{\pi_{C''}(\HL), \pi_{B''}(\VL), \pi_{B''}\pi_{C''}(\Mix)} 
          \subset \widetilde{E'} \otimes C' + B'\otimes F'.
    \]
    That is, $W'$ is $({\bf{\widetilde{e}'}}, \bff')$-hook shaped.
    Since $\bff'=2$, Proposition~\ref{prop_1_2_hook_shaped} 
       shows that $\hl \ge {\bf{\widetilde{e}'}}\ge 2$.
    Similarly, $\hr$, $\vl$, $\vr$ are also at least $2$.
    We also see that $\widetilde{E'} = E'$, that is, the elements of type $\HL$ are concise in $E'$.
    
    Next, we show that $\vl \ne 2$, which is perhaps the most interesting part of this example.
    For this purpose we consider the projection $\pi_{E'\oplus B''}\colon B \to B'/E'$.
    The related map $B\otimes C \to   (B'/E')\otimes C$ 
       (which by the standard abuse we also denote $\pi_{E'\oplus B''}$), 
       kills all the rank one tensors of types $\HL$, $\HR$, $\VR$ and $\Mix$, 
       leaving only those of type $\VL$ alive.
    The image $\pi_{E'\oplus B''}(W) \subset  (B'/E')\otimes F'$ has rank at most $\vl$ 
       and is concise (otherwise, either Proposition~\ref{prop_1_2_hook_shaped} shows the additivity or $p'$ is not concise, a contradiction in both cases).
    Note that $(B'/E')\otimes F' \simeq \CC^2 \otimes \CC^2$ and there are only two (up to a change of basis) 
       concise linear subspaces of $\CC^2 \otimes \CC^2$ which have rank at most $2$.
    In both cases it is straightforward to verify that there exists $\beta'\in (B'/E')^* \subset (B')^*$ 
       such that $\beta'(p) = \beta'(p') \in A' \otimes C'$ has rank $1$.
    Then, by swapping the roles of $A$ and $B$, the process of repletion and digestion 
       (Lemma~\ref{lemma_can_digest}) leads to a smaller tensor which is also a counterexample to the additivity of the rank, again a contradiction.
    Thus $R(\pi_{E'\oplus B''}(W))$ must be at least $3$ and consequently, $\vl \ge 3$. 
    The same argument shows that $\vr \ge 3$.
    
    Combining the inequalities obtained so far we also get:
    \[
      \mix = 13 - (\hl + \hr + \vl + \vr) \le 3.
    \]
    The inequality $\mix \ge 1$ follows from 
       Corollary~\ref{cor_inequalities_HL__Mix}\ref{item_mix_at_least_1},
       and it is left to show only the last two inequalities. 
    To prove  $\hl+\vl \le 6$,
       we use Proposition~\ref{prop_projection_inequalities}\ref{item_projection_inequality_short_bis}:
    \[
      7\le \hr +\vr +\mix = R(W) - (\hl+\vl) = 13-(\hl+\vl).
    \]
    The last inequality follows from a similar argument.
    \end{proof}

\section{Additivity of the tensor border rank}\label{sect_border_rank}

Throughout this section we will follow Notations~\ref{notation} and~\ref{notation_2}. Moreover, we restrict to the base field $\kk=\CC$.

We turn our attention to the additivity of the border rank.
That is, we ask for which tensors $p'\in A'\otimes B'\otimes C'$ and  $p''\in A''\otimes B''\otimes C''$
  the following equality holds:
  \[
     \br(p'\oplus p'') = \br(p')  + \br(p'').
  \]
Since the known counterexamples to the additivity are much smaller than in the case of the additivity 
   of the tensor rank, our methods are more restricted to very small cases.
We commence with the following elementary observation.

\begin{lem}\label{lemma_additivity_small_brank}
Consider concise tensors $p'\in A'\ts B'\ts C'$ and $p''\in A''\ts B''\ts C''$ with  $\br(p')\le \bfa'$ and $\br(p'')\le \bfa''$ (thus in fact $\br(p') = \bfa'$ and $\br(p'')= \bfa''$).
Let $p=p' \oplus p''$.
Then the additivity of the border rank holds $\br(p) =  \br(p')+\br(p'')$. 
\end{lem}
\begin{proof}
Since $p'$ and $p''$ are concise, the linear maps 
   $p'\colon (A')^* \to B'\otimes C'$ and $p''\colon (A'')^* \to B''\otimes C''$ are injective.
Then also the map $p\colon A^\ast\to B\ts C$ is injective and 
    \[
      \br(p) \ge \dim p(A^*) = \dim p'((A')^*) + \dim p''((A'')^*) = \br(p') + \br(p'').
    \]
The opposite inequality always holds.
\end{proof}

\begin{cor}\label{cor_additivity_brank_very_small_a_b_c}
Suppose both triples of integers $(\bfa',\bfb',\bfc')$ and $(\bfa'',\bfb'',\bfc'')$ 
   fall into one of the following cases:
 $(a,b,1)$, $(a,1,c)$,
 $(a,b,2)$ with $a\ge b \ge 2$,
 $(a,2,c)$ with $a\ge c \ge 2$,
 $(a,b,c)$ with $a\ge b c$.
Then for any concise tensors $p'\in A'\otimes B'\otimes C'$ and $p''\in A''\otimes B''\otimes C''$ 
   the additivity of the border rank holds.
\end{cor}
Note that the list of triples in the corollary is a bit exaggerated, as some of these triples have no concise tensors. 
However, this phrasing is convenient for further applications and search for unsolved pairs of triples.
\begin{proof}
 After removing the triples that do not admit any concise tensor the list reduces to:
 $(a,a,1)$, $(a,1,a)$,
 $(a,b,2)$ (for $2 \le b \le a \le 2b$), 
 $(a,2,c)$ (for $2 \le c \le a \le 2c$),
 $(bc,b,c)$.
We claim that in all these cases $\br(p') = \bfa'$ and $\br(p'') = \bfa''$. In fact:
\begin{itemize}
  \item The claim is clear for  $(a,1,a)$, $(a,a,1)$, and $(bc,b,c)$.
  \item For $(a,a,2)$ and $(a,2,a)$ the claim follows from the classification of such tensors, 
          see the argument in the first paragraph of \cite[\S5.3]{jabu_han_mella_teitler_high_rank_loci}.
  \item For $(a,b,2)$ (with $2 \le b < a \le 2b$),  and $(a,2,c)$ (with $2 \le c < a \le 2c$),
           the claim follows from the previous case: any such concise tensor $T$ has border rank at least $a$.
        But $T$ is  at the same time a (non-concise) tensor in a larger tensor space 
           $\CC^a \otimes \CC^a \otimes \CC^2$ or $\CC^a \otimes \CC^2 \otimes \CC^a$.
        Thus by Lemma~\ref{lem_rank_independent_of_ambient} the border rank of $T$ 
           is at most the generic (border) rank 
           in this larger space, which is equal to $a$ by the previous item.
\end{itemize}
Therefore we conclude using Lemma~\ref{lemma_additivity_small_brank}.
\end{proof}

Theorem~\ref{thm_additivity_border_rank_in_dimension_4} claims that the additivity of the border rank holds for $\bfa,\bfb,\bfc\le 4$.
Most of the cases follow from Corollary~\ref{cor_additivity_brank_very_small_a_b_c}, 
   with the exception of $(3+1,2+2,2+2)$ and $(3+1,3+1,3+1)$, 
   which are covered in Sections~\ref{sec_br_3_2_2_1_b_c} and~\ref{sec_br_3_3_3_1_b_c}.

\begin{defin}
   Assume $p, q \in A\otimes B\otimes C$ are two tensors. 
   We say that $p$ is \emph{more degenerate} than $q$ if $p \in \overline{GL(A)\times GL(B) \times GL(C)\cdot q}$.
\end{defin}

\begin{example}\label{example_222_degenerates_to_122}
   Any concise tensor in $\CC^1\otimes \CC^2\otimes \CC^2$ is more degenerate
      than any concise tensor in $\CC^2\otimes \CC^2\otimes \CC^2$.
\end{example}

\begin{example}\label{example_degeneration_in_322}
   Consider concise tensors in $\CC^3\times \CC^2\times \CC^2$.
   According to \cite[Table 10.3.1]{landsberg_tensorbook}, 
      there are two orbits of  the action of $GL_3\times GL_2\times GL_2$ of such tensors, 
      both orbits of border rank $3$. 
   One orbit is ``generic'', the other is \emph{more degenerate}.
   The latter is represented by:
   \[
      p = a_1 \otimes b_1 \otimes c_1 
        + a_2 \otimes b_1 \otimes c_2 
        + a_3 \otimes b_2 \otimes c_1.
   \]
\end{example}

\begin{lem}\label{lemma_more_degenerate}
   Suppose $p'\in A'\otimes B'\otimes C'$ is an arbitrary tensor 
     and $p'', q'' \in A''\otimes B''\otimes C''$ are  
     such that $\br (p'')=\br(q'')$ and $p''$ is more degenerate than $q''$.
   If the additivity of the border rank holds for $p'\oplus p''$ 
     then it also holds for $p'\oplus q''$. 
\end{lem}
\begin{proof}
   Since $p''$ is more degenerate than $q''$ also $p'\oplus p''$ is more degenerate than $p'\oplus q''$.
   Thus 
   \[
      \br(p'\oplus q'') \ge  \br(p'\oplus p'') = \br(p') + \br(p'')= \br(p') + \br(q'').
   \]
\end{proof}

\subsection{Strassen's equations of secant varieties}\label{sec_Strassen_equations}

Often as a criterion to determine whether a tensor is or is not of a given border rank, we exploit
defining equations of the corresponding secant varieties. 
We review here one type of equations that is most important for the small cases we consider in this article.

First assume $\bfb=\bfc$ and consider the space of square matrices $B\ts C$.
Let $f_\bfb:(B\ts C)^{\times3}\to B\ts C$ be the map of matrices defined as follows:
\begin{equation}\label{equation_adjoint}
f_\bfb(x,y,z)=x\adj (y) z-z \adj (y)x,
\end{equation}
where $\adj(y)$ denotes the adjoint matrix of $y$.

As in Section~\ref{section_slices} write 
$$
p=\sum_{i=1}^{\bfa} a_i\ts w_i,
$$
where $w_1,\dots,w_{\bfa}\in W:=p({A}^\ast)\subset B\ts C$ are $\bfb\times\bfc$ matrices
and $\{a_1,\dots,a_\bfa\}$ is a basis of $A$.

\begin{prop}\label{prop_strassen_adj}
   Assume that $p\in A\ts B\ts C$. 
   \begin{enumerate}
      \item \cite{strassen_equations} \label{item_strassen_adj_3}
            Suppose $\bfa=\bfb=\bfc=3$.
            Then  $\br(p)\le3$ if and only if $f_3(x,y,z)=\underline{0}$ for every $x,y,z\in W$.
      \item \cite{landsbergmanivel08} \label{item_strassen_adj_more}
            Suppose  $\bfa=\bfb=\bfc$ and $\br(p)\le\bfa$.
            Then $f_\bfa(x,y,z)=\underline{0}$, for every $x,y,z\in W$.
\end{enumerate}
\end{prop}
See also \cite[Thm~3.2]{friedland_salmon_problem_paper}.

We also recall Ottaviani's derivation of Strassen's equations (\cite{ottaviani_symplectic_bundles_secants_Luroth}, see also \cite[Sect.~3.8.1]{landsberg_tensorbook}) for secant varieties of three factor Segre embeddings.

Given a tensor $p:B^\ast\to A\ts C$,  consider the contraction operator
$$
p^{\wedge}_A:A\ts B^\ast\to\Lambda^2 A\ts C,
$$
obtained as composition of the map $ \Id_A\ts p:A\ts B^\ast\to A^{\ts2}\ts C$ with the natural  projection  $A^{\ts2}\ts C\to\Lambda^2 A\ts C$.

\begin{prop}[{\cite[Theorem 4.1]{ottaviani_symplectic_bundles_secants_Luroth}}]
\label{prop_ottaviani_derivation_strassen}
Assume $3\le\bfa\le \bfb,\bfc$.
If $\br(p)\le r$, then $\rk({p}^\wedge_{A})\le r(\bfa-1)$.
\end{prop}

If $\bfa=3$, we can slice $p$ as follows (cf. Section~\ref{section_slices}):
$p=\sum_{i=1}^3a_i \ts w_i\in A\ts B\ts C$, with $w_i\in B\ts C$. 
Then the matrix representation of ${p}^\wedge_{A}$  in block matrices is the following $(\bfb+\bfb+\textbf{b},\bfc+\bfc+\bfc)$ partitioned matrix
\begin{equation}\label{equation_contraction_operator}
M_3(w_1,w_2,w_3):=\left(
\begin{array}{ccc}
\underline{0}&w_3&-w_2\\
-w_3&\underline{0}&w_1\\
w_2&-w_1&\underline{0}
\end{array}\right).
\end{equation}

\begin{prop}[{\cite[Prop.~7.6.4.4]{landsberg_tensorbook}}]\label{prop_strassen_degree9}
If $\bfa=\bfb=\bfc=3$, the degree nine equation 
\[
   \det({p}^\wedge_{A})=0
\]
defines  the variety $\sigma_4(\PP A\times \PP B\times\PP C)\subset\PP(A\ts B\ts C)$.
\end{prop}

If $\bfa=4$ and $p=\sum_{i=1}^4a_i \ts w_i\in A\ts B\ts C$, with 
$w_i\in B\ts C$, then the matrix representation of ${p}^\wedge_{A}$  
in block matrices is the following $(4\cdot\bfb,6\cdot\bfc)$ 
% \eli{it should be $(\bfb+\bfb+\bfb+\bfb,\bfc+\bfc+\bfc+\bfc+\bfc+\bfc)$ but it is a bit long... Is it ok to leave as it is?}
partitioned matrix 
\begin{equation}\label{equation_contraction_operator_4}
M_4(w_1,w_2,w_3,w_4):=\left(
\begin{array}{cccccc}
\underline{0}&w_3&-w_2&w_4&\underline{0}&\underline{0}\\
-w_3&\underline{0}&w_1&\underline{0}&-w_4&\underline{0}\\
w_2&-w_1&\underline{0}&\underline{0}&\underline{0}&w_4\\
\underline{0}&\underline{0}&\underline{0}&-w_1&w_2&-w_3\\
\end{array}\right).
\end{equation}

\subsection{Case \texorpdfstring{$(3+1,2+\bfb'',2+\bfc'')$}{(3+1,2+b'',2+c'')}}\label{sec_br_3_2_2_1_b_c}

Assume $\bfa'=3$, $\bfb'=\bfc'=2$ and $\bfa''=1$.

\begin{prop}\label{prop_br_case_3_2_2_1_2_2}
  For any $p'\in \CC^3\otimes\CC^2\otimes \CC^2$ and $p''\in \CC^1\otimes\CC^{\bfb''}\otimes \CC^{\bfc''}$
     the additivity of the border rank holds.
\end{prop}
\begin{proof}
  We can assume $p''$ is concise, so that $\br(p'')=\bfb''=\bfc''$.
  Also if $p'$ is not concise, then Corollary~\ref{cor_additivity_brank_very_small_a_b_c} shows the claim.
  So suppose $p'$ is concise and thus $\br(p')=3$.

We can write $p'=a_1\otimes w'_1+a_2\otimes w'_2+a_3\otimes w'_3$ and $p''=a_4\otimes w''_4$, where $w'_1\dots, w'_3$ are $2\times2$ matrices and $w_4''$ is an invertible $\bfb''\times \bfb''$ matrix.

As for $p'$, by Example~\ref{example_degeneration_in_322} and Lemma~\ref{lemma_more_degenerate} 
  we can choose the more degenerate tensor, which has the following normal form:
$$
 w'_1=\left(\begin{array}{cc} 1&0\\0&0\end{array}\right),
 w'_2=\left(\begin{array}{cc} 0&1\\0&0\end{array}\right),
 w'_3=\left(\begin{array}{cc} 0&0\\1&0\end{array}\right).
$$
Write $p=\sum_{i=1}^4 a_i\ts w_i$, where $w_i$ are the following $(2+\bfb'',2+\bfb'')$ partitioned matrices 
$$
w_i=\left(\begin{array}{cc} w'_i&\underline{0}\\ \underline{0}&\underline{0} \end{array}\right), i=1,2,3, \ 
w_4=\left(\begin{array}{cc} \underline{0}&\underline{0}\\
 \underline{0}&w''_4 \end{array}\right).
$$

We use the same notation as in Section~\ref{sec_Strassen_equations}. 
We claim that the matrix representing the contraction operator $p^\wedge_A$, 
   denoted by $M_4(w_1,w_2,w_3,w_4)$ as in \eqref{equation_contraction_operator_4}, has rank $7+3\bfb''$. 
We conclude that $\br(p) \ge 3+\bfb''=\br(p')+\br(p'')$
   by Proposition~\ref{prop_ottaviani_derivation_strassen} showing the addivitity.

In order to prove the claim, we observe that $M_4(w_1,w_2,w_3,w_4)$ can be transformed via permutations of rows and columns into the following $(6+3\bfb''+2+\bfb'',6+3\bfb''+2+2+2+3\bfb'')$-partitioned matrix
$$
\left(
\begin{array}{cccccc}
M_3(w'_1,w'_2,w'_3)&\underline{0}&\underline{0}&\underline{0}&\underline{0}&\underline{0}\\
\underline{0}&N&\underline{0}&\underline{0}&\underline{0}&\underline{0}\\
\underline{0}&\underline{0}&-w'_1&w'_2&-w'_3&\underline{0}\\
\underline{0}&\underline{0}&\underline{0}&\underline{0}&\underline{0}&\underline{0}\\
\end{array}\right),
$$
where $N$ is the following $3\bfb''\times3\bfb''$ matrix
$$
N=\left(\begin{array}{ccc}
w_4''&\underline{0}&\underline{0}\\
\underline{0}&-w_4''&\underline{0}\\
\underline{0}&\underline{0}&w_4''
\end{array}\right).
$$
One can compute that the rank of  $M_3(w'_1,w'_2,w'_3)$ equals $5$. 
Moreover, since $\rk(N)=3\bfb''$ and $\rk((-w'_1,w'_2,-w'_3))=2$, we conclude the proof of the claim.
\end{proof}

\subsection{Case \texorpdfstring{$(3+1,3+\bfb'',3+\bfc'')$}{(3+1,3+b'',3+c'')}}\label{sec_br_3_3_3_1_b_c}

Recall our usual setting: $p'\in A'\otimes B'\otimes C'$, $p''\in A''\otimes B''\otimes C''$,
$\bfa':=\dim A'$, etc.  (Notation~\ref{notation_2}).
In this subsection we are going to prove the following case of additivity of the border rank.
\begin{prop}\label{prop_br_case_3_3_3_1_1_1}
   The additivity of the border rank holds for $p'\oplus p''$ if $\bfa'=\bfb'=\bfc'=3$, and 
      $p'$ is concise and $\bfa''=1$.
\end{prop}
\begin{proof}
   By replacing $B''$ and $C''$ with smaller spaces
      we can assume $p''$ is also concise and in particular $\bfb''=\bfc''$.
   If $\br(p') = 3$ then Lemma~\ref{lemma_additivity_small_brank} implies the claim. 
   On the other hand, by Terracini's Lemma, $\br(p')\le 5$.
   Thus it is sufficient to treat the cases $\br(p')=4$ and $\br(p')=5$.

   Let $\{a_1,a_2,a_3\}$ be a basis of $A'$ and let $\{a_4\}$ be a basis of $A''\simeq \CC$. 
   Write 
    \begin{equation}\label{equation_p'_(3,3,3)}
       p'=a_1\ts w'_1+a_2\ts w'_2+a_3\ts w'_3,
    \end{equation}
   where $w'_1,w'_2,w'_3\in W':=p'((A')^\ast)\subset B'\ts C'$ are $3\times3$ matrices.
   Similarly, let 
   $$
      p=a_1\ts w_1+a_2\ts w_2+a_3\ts w_3+a_4\ts w_4,
   $$
   where $w_1,w_2,w_3,w_4\in W:=p({A}^\ast)\subset B\ts C$ are $(3+\bfb'',3+\bfb'')$ partitioned matrices:
   \begin{equation}\label{equation_slicing_border}
      w_i=\left(
      \begin{array}{cc}
         w'_i          & \underline{0} \\
         \underline{0} &            0  \\
      \end{array}\right), \ i=1,2,3, \text{ and }
      w_4=\left(
      \begin{array}{cc}
        \underline{0}  & \underline{0}\\
        \underline{0}  &            w_4''
      \end{array}\right).
    \end{equation}
    
We now analyse the two cases $\br(p') =4$ and $\br(p') =5$ separately.
    
\paragraph{The additivity holds if the border rank of $p'$ is equal to four}
Assume by contradiction that $\br(p)\le \bfb''+3 = \br(p')+\br(p'') -1$.
By Proposition~\ref{prop_strassen_adj}\ref{item_strassen_adj_more}, 
   we obtain the following equations: $f_{\bfb''+3}(x',y'+y'',z')=\underline{0}$,
   for every $x',y',z'\in W'= p'\left((A')^*\right)$ and $0\ne y''\in  W''= p''\left((A'')^*\right)$.
We can see that $\adj(y'+y'')$ is the following $(3+\bfb'',3+\bfb'')$ partitioned matrix
$$
\adj(y'+y'')=
\left(
\begin{array}{cc}
\det(y'') \adj (y')& \underline{0} \\
\underline{0}& \det(y')\adj(y'')\\
\end{array}\right).
$$
Therefore we have 
$$
x' \adj(y'+y'') z' =
\left(
\begin{array}{cc}
\det(y'') x'\adj (y')z'& \underline{0} \\
\underline{0}& 0 \\
\end{array}\right).
$$
Since $p''$ is concise, $\det(y'')\ne 0$, and 
   thus from the vanishing of $f_{\bfb''+3}(x', y'+y'', z')$ we also obtain that $f_3(x',y',z')=0$.
Therefore $\br(p')\le 3$ by Proposition~\ref{prop_strassen_adj}\ref{item_strassen_adj_3}, a contradiction.

\paragraph{The additivity holds if the border rank of $p'$ is equal to five}

Consider the projection $\pi: A\ts B\ts C \to A'\ts B\ts C$ given by 
\begin{align*}
a_i&\mapsto a_i,\ i=1,2,3\\
a_4&\mapsto a_1+a_2+a_3.
\end{align*}  
Consider $\bar{p}:=\pi(p)\in A'\ts B\ts C$  and write $\bar{p}=a_1\otimes \bar{w}_1+a_2\otimes \bar{w}_2+a_3\otimes \bar{w}_3$, where, for $i=1,2,3$, 
$\bar{w}_i$ is the $(3+\bfb'',3+\bfb'')$ partitioned matrix
 $$
\bar{w}_i=\left(\begin{array}{cc}
w'_i&0\\
0&w''_4
\end{array}\right).
$$
We claim that $\rk({\bar{p}}^\wedge_{A'})=9+2\bfb''$.
Indeed, by swapping both rows and columns of $M_3(\bar{w}_1,\bar{w}_2,\bar{w}_3)$ 
   (see Equation~\ref{equation_contraction_operator}) 
   we obtain the following $(9+3\bfb'',9+3\bfb'')$ partitioned matrix
$$
\left(\begin{array}{cc}
{p'}^\wedge_{A'}&\underline{0}\\
\underline{0}&M_3(w_4'',w_4'',w_4'')
\end{array}\right).
$$
Since $\br(p')=5$, the matrix ${p'}^\wedge_{A'}$ has rank $9$, by Proposition~\ref{prop_strassen_degree9}. Moreover  $M_3(w_4'',w_4'',w_4'')$ has rank $2\bfb''$.
Therefore, by Proposition~\ref{prop_ottaviani_derivation_strassen}, we obtain $\br(\bar{p})\ge 5+ \bfb''$.
We conclude by observing that $\br(p)\ge \br(\bar{p})$.
\end{proof}

This concludes the proof of Theorem~\ref{thm_additivity_border_rank_in_dimension_4}, as all possible splittings 
$\bfa= \bfa'+\bfa''$, $\bfb= \bfb'+\bfb''$, $\bfc= \bfc'+\bfc''$ with $\bfa, \bfb,\bfc\le 4$
  are covered either by Corollary~\ref{cor_additivity_brank_very_small_a_b_c}
  or one of Propositions~\ref{prop_br_case_3_2_2_1_2_2} or \ref{prop_br_case_3_3_3_1_1_1}.

   \begin{table}[htb]
   \[
    \begin{array}{|c|c|c|c|c|}
    \hline
     \# & (\bfa', \bfb', \bfc') & (\bfa'', \bfb'', \bfc'') & \br(p') & \br(p'') \\
     \hline
%  1. &  3, 2, 2 &  1, 3, 3 &        3 & 3 \\
 1. &  3, 2, 2 &  2, 3, 2 &        3 & 3 \\
 2. &  3, 3, 2 &  2, 2, 3 &        3 & 3 \\
%  4. &  3, 3, 3 &  2, 2, 1 &     4, 5 & 2 \\
 3. &  3, 3, 3 &  2, 2, 2 &     4, 5 & 2 \\
 4. &  4, 2, 2 &  1, 2, 2 &        4 & 2 \\
 5. &  4, 2, 2 &  1, 3, 3 &        4 & 3 \\
 6. &  4, 3, 2 &  1, 2, 2 &        4 & 2 \\
 7. &  4, 3, 3 &  1, 1, 1 &        5 & 1 \\
 8. &  4, 3, 3 &  1, 2, 2 &        5 & 2 \\
 9. &  4, 4, 3 &  1, 1, 1 &     5, 6 & 1 \\
10. &  4, 4, 4 &  1, 1, 1 &  5, 6, 7 & 1 \\
\hline
    \end{array}
   \]
    \caption{The list of pairs of concise tensors and their border ranks that should be checked to determine 
                the additivity of the border rank for $\bfa, \bfb, \bfc \le 5$. 
             This list contains all pairs of concise tensors not covered by 
                 Corollary~\ref{cor_additivity_brank_very_small_a_b_c}, 
                or Proposition~\ref{prop_br_case_3_2_2_1_2_2},
                or Proposition~\ref{prop_br_case_3_3_3_1_1_1},
                together with their possible border ranks,
                excluding the cases covered by Lemma~\ref{lemma_additivity_small_brank}.   
             The maximal possible values of border ranks above have been obtained 
                from~\cite[Sect.~4]{abo_ottaviani_peterson_segre}.}   \label{tab_cases_for_br_up_to_dim_5}
\end{table}

One could analyse the additivity for $\bfa, \bfb, \bfc \le 5$ 
   (so for the bound one more than in Theorem~\ref{thm_additivity_border_rank_in_dimension_4}) 
by checking all 10 possible cases listed in Table~\ref{tab_cases_for_br_up_to_dim_5}.
We conclude the article by solving also Case $3$ from the table.
\begin{example}
    If $p'\in \CC^3\otimes \CC^3\otimes \CC^3$ and $p''\in \CC^2\otimes \CC^2\otimes \CC^2$ are both concise, 
      then the additivity of the border rank holds for $p'\oplus p''$.
    Indeed, by Example~\ref{example_222_degenerates_to_122} 
      there exists $q''\in \CC^1\otimes \CC^2\otimes \CC^2$ more degenerate than $p''$, 
      but of the same border rank. 
    By Lemma~\ref{lemma_more_degenerate} it is enough to prove the additivity for $p'\oplus q''$. 
    This is provided by Proposition~\ref{prop_br_case_3_3_3_1_1_1}.
\end{example}

\section*{Acknowledgments}

We are enormously grateful to Joseph Landsberg for introducing us 
   to this topic and numerous discussions and explanations. 
We also thank Michael Forbes, Mateusz Micha{\l}ek, Artie Prendergast-Smith, Zach Teitler, and Alan Thompson 
   for reference suggestions and their valuable comments.
We are also greatful the referees and the journal editors for their suggestions that have helped to improve the results and presentation.

The research on this project was spread across a wide period of time.
It commenced once E.~Postinghel was a postdoc at IMPAN in Warsaw (Poland, 2012-2013) under the project ``Secant varieties, computational complexity, and toric degenerations'' realised within the Homing Plus programme of Foundation for Polish Science, cofinanced from European Union, Regional Development Fund.

Also our collaboration in years 2014-2019 was possible during many meetings, 
   in particular those that were related to special programmes, such as:
   the thematic semester ``Algorithms and Complexity in Algebraic Geometry'' at 
      Simons Institute for the Theory of Computing (2014), 
   the Polish Algebraic Geometry mini-Semester (miniPAGES, 2016),
   and the thematic semester Varieties: Arithmetic and Transformations (2018).
   The latter two events were supported by the grant 346300 for IMPAN
      from the Simons Foundation and the matching 2015-2019 Polish MNiSW fund.
   We are grateful to the participants of these semesters for numerous inspiring discussions 
      and to the sponsors for supporting our participations.
      
     We are grateful to Loughborough University for hosting our collaboration in May 2017 and to Copenhagen University for hosting us in January 2019.
     
In addition, J.~Buczy{\'n}ski is supported by 
    the Polish National Science Center project ``Algebraic Geometry: Varieties and Structures'', 
       2013/08/A/ST1/00804, 
    the scholarship ``START'' of the Foundation for Polish Science 
    and a scholarship of Polish Ministry of Science. 
    
    E.~Postinghel was supported by a grant of the Research Foundation - Flanders (FWO) between 2013-2016 and is supported by the EPSRC grant no.~EP/S004130/1 from late 2018.

Finally, the paper is also a part of the activities of the AGATES research group.

\bibliography{references}
\bibliographystyle{siamplain}

\end{document}